\documentclass[a4paper,reqno]{amsart}

\usepackage{a4wide}
\usepackage[ansinew]{inputenc}
\usepackage[english,ngerman]{babel}
\usepackage{graphicx}
\usepackage{mathabx}
\usepackage{subfigure}
\usepackage{listings}
\usepackage{mathtools}
\usepackage{nicefrac}
\usepackage{float}
\usepackage{geometry}
\usepackage{dsfont}
\usepackage{amssymb}
\usepackage{stmaryrd}
\usepackage{caption}	

\title[Long-time behaviour for a family of fourth order]{Long-time behaviour of a fully discrete Lagrangian scheme for a family of fourth order}

\author{Horst Osberger}
\address{Horst Osberger \\ Zentrum Mathematik \\ TU M\"unchen \\ Boltzmannstr. 3 \\ D-85748 Garching \\ Germany}
\email{osberger@ma.tum.de}

\thanks{This research was supported by the DFG Collaborative Research Center TRR 109, ``Discretization in Geometry and Dynamics''.}

\begin{document}

\newcommand{\setR}{\mathbb{R}}
\newcommand{\setN}{\mathbb{N}}
\newcommand{\dd}{\,\mathrm{d}}
\newcommand{\indy}{\mathbb{I}}
\newcommand{\eins}{\mathds{1}}

\newcommand{\theX}{\mathrm{X}}
\newcommand{\theZ}{\mathrm{Z}}
\newcommand{\theU}{\mathrm{U}}
\newcommand{\xvec}{\vec{\mathrm x}}
\newcommand{\yvec}{\vec{\mathrm y}}
\newcommand{\zvec}{\vec{\mathrm z}}
\newcommand{\vvec}{\vec{\mathrm v}}
\newcommand{\wvec}{\vec{\mathrm w}}
\newcommand{\spr}[2]{\left\langle#1,#2\right\rangle_\delta}
\newcommand{\nrm}[1]{\left\|#1\right\|_\delta}
\newcommand{\slp}[1]{\left|\partial_\theh #1\right|}
\newcommand{\rvec}{\vec{\mathrm \varphi'}}
\newcommand{\gvec}{\vec{g}}
\newcommand{\s}{{\widetilde{\xi}}_k}

\newcommand{\convf}{\mathbf{u}}
\newcommand{\cf}{\mathbf{u}}
\newcommand{\cfM}{\mathbf{w}}
\newcommand{\cz}{\mathbf{z}}
\newcommand{\cX}{\mathbf{X}}
\newcommand{\cY}{\mathbf{Y}}

\newcommand{\theh}{{\delta}}
\newcommand{\thet}{{\boldsymbol \tau}}
\newcommand{\dm}{{\overline{\delta}}}
\newcommand{\xspc}{\mathfrak{X}}
\newcommand{\xseq}{\mathfrak{x}}
\newcommand{\xseqN}{\xseq_\theh}
\newcommand{\zseq}{\mathfrak{z}}
\newcommand{\dens}{\mathcal{P}(\Omega)}
\newcommand{\densN}{\mathcal{P}_\theh(\Omega)}
\newcommand{\densNs}{\mathcal{P}_\theh^*(\Omega)}
\newcommand{\densQ}{\mathcal{P}_K^{\operatorname{quad}}(\Omega)}
\newcommand{\Lloc}{L_{\operatorname{loc}}}

\newcommand{\tv}[1]{\{#1\}_{\mathrm{TV}}}
\newcommand{\ti}[1]{\left\lbrace#1\right\rbrace_{\tau}}
\newcommand{\tti}[1]{\left\langle#1\right\rangle_{\tau}}

\newcommand{\Wmat}{\mathrm{W}}
\newcommand{\Wmatn}{\widetilde{\mathrm{W}}}

\newcommand{\I}{\{1,2,\ldots,K\}}
\newcommand{\Ih}{\{\tfrac{1}{2},\tfrac{3}{2},\ldots,\tfrac{2K-1}{2}\}}
\newcommand{\II}{\mathbb{I}_K}

\newcommand{\ival}{{\mathbb{I}_K^0}}
\newcommand{\hval}{{\mathbb{I}_K^{1/2}}}
\newcommand{\kmh}{{k-\frac{1}{2}}}
\newcommand{\kph}{{k+\frac{1}{2}}}
\newcommand{\kpd}{{k+\frac{3}{2}}}
\newcommand{\kmd}{{k-\frac{3}{2}}}
\newcommand{\jmh}{{j-\frac{1}{2}}}
\newcommand{\jph}{{j+\frac{1}{2}}}
\newcommand{\jmd}{{j-\frac{3}{2}}}
\newcommand{\kpmh}{{k\pm\frac{1}{2}}}
\newcommand{\kmph}{{k\mp\frac{1}{2}}}
\newcommand{\kappm}{{\kappa-\frac12}}
\newcommand{\kappp}{{\kappa+\frac12}}

\newcommand{\Kmh}{{K-\frac{1}{2}}}
\newcommand{\imh}{{\frac{1}{2}}}
\newcommand{\iph}{{\frac{3}{2}}}
\newcommand{\Kmd}{{K-\frac{3}{2}}}

\newcommand{\thegrad}{\operatorname{grad}_{\wass}}
\newcommand{\eps}{{\varepsilon}}
\newcommand{\wass}{\mathcal{W}_2}
\newcommand{\wassN}{\mathbf{W}_2}
\newcommand{\grad}{\partial_{\xvec}}
\newcommand{\wgrad}{\nabla_\delta}
\newcommand{\prss}{\mathrm{P}_\alpha}
\newcommand{\nci}{\mathfrak{F}}
\newcommand{\vel}{\mathbf{v}_\alpha}

\newcommand{\baru}{\bar{u}}
\newcommand{\hatz}{\widehat{z}}
\newcommand{\hatu}{\widehat{u}}

\newcommand{\Fal}{\mathcal{F}_{\alpha,\lambda}}
\newcommand{\Fhl}{\mathcal{F}_{1/2,\lambda}}
\newcommand{\Fan}{\mathcal{F}_{\alpha,0}}
\newcommand{\Fhn}{\mathcal{F}_{1/2,0}}
\newcommand{\Falzu}{\widetilde{\mathbf{F}}_{\alpha,\lambda}}
\newcommand{\Falz}{\mathbf{F}_{\alpha,\lambda}}
\newcommand{\Fhlz}{\mathbf{F}_{1/2,\lambda}}
\newcommand{\Fanz}{\mathbf{F}_{\alpha,0}}
\newcommand{\Fhnz}{\mathbf{F}_{1/2,0}}
\newcommand{\Fy}{\mathbf{F}_{\alpha}}

\newcommand{\Hal}{\mathcal{H}_{\alpha,\lambda}}
\newcommand{\Hhl}{\mathcal{H}_{1/2,\lambda}}
\newcommand{\Han}{\mathcal{H}_{\alpha,0}}
\newcommand{\Hhn}{\mathcal{H}_{1/2,0}}
\newcommand{\Halz}{\mathbf{H}_{\alpha,\lambda}}
\newcommand{\Hhlz}{\mathbf{H}_{1/2,\lambda}}
\newcommand{\Hanz}{\mathbf{H}_{\alpha,0}}
\newcommand{\Hhnz}{\mathbf{H}_{1/2,0}}
\newcommand{\Haiz}{\mathbf{H}_{\alpha,1}}
\newcommand{\Halzd}[1]{\mathcal{H}_{\alpha,\lambda}^\delta(#1)}

\newcommand{\olHal}{\overline{\Hal}}
\newcommand{\olFal}{\overline{\Fal}}
\newcommand{\Vz}{\mathbf{V}}

\newcommand{\fa}{f_\alpha}
\newcommand{\fh}{f_{1/2}}
\newcommand{\La}{\Lambda_{\alpha,\lambda}}
\newcommand{\Lh}{\Lambda_{1/2,\lambda}}
\newcommand{\dela}{\delta_\alpha}
\newcommand{\phia}{\varphi_\alpha}

\newcommand{\psia}{\psi_\alpha}
\newcommand{\psih}{\psi_{\frac{1}{2}}}

\newcommand{\hatf}{\theta}
\newcommand{\D}{\operatorname{D}_\theh}
\newcommand{\ee}{\mathbf{e}}
\newcommand{\intom}{\int_\Omega}

\newcommand{\uinf}{u^{\infty}}
\newcommand{\vinf}{v^{\infty}}
\newcommand{\xvecm}{\xvec_{\theh}^{\operatorname{min}}}
\newcommand{\um}{u_\delta^{\operatorname{min}}}
\newcommand{\umh}{\hat{u}_\delta^{\operatorname{min}}}
\newcommand{\bal}{\mathrm{b}_{\alpha,\lambda}}
\newcommand{\bhl}{\mathrm{b}_{1/2,\lambda}}
\newcommand{\ds}{\mathfrak{d}}
\newcommand{\bant}{\mathrm{b}_{\alpha,0}^*}
\newcommand{\bai}{\mathrm{b}_{\alpha,1}}
\newcommand{\taut}{\widehat{\tau}}
\newcommand{\lt}{\widehat{\lambda}}
\newcommand{\bantn}{\mathrm{b}_{\Delta,\alpha,0}^n}
\newcommand{\St}{S_\tau}
\newcommand{\bvec}{\vec{\mathrm b}}
\newcommand{\sh}{\hat{s}}
\newcommand{\Deltat}{{\widehat{\Delta}}}
\newcommand{\thett}{{{\boldsymbol \tau}_{\taut}}}

\newcommand{\Halzm}{\Halz^{\operatorname{min}}}
\newcommand{\Hanzm}{\Hanz^{\operatorname{min}}}
\newcommand{\Falzm}{\Falz^{\operatorname{min}}}

\newcommand{\vech}{{\vec\delta}}
\newcommand{\xseqNN}{\xseq_\vech}
\newcommand{\densNN}{\mathcal{P}_\vech(\Omega)}

\newcommand{\essinf}{\operatorname*{ess\,inf}}
\newtheorem{thm}{Theorem}
\newtheorem{prp}[thm]{Proposition}
\newtheorem{lem}[thm]{Lemma}
\newtheorem{cor}[thm]{Corollary}
\newtheorem{rmk}[thm]{Remark}
\newtheorem{dfn}[thm]{Definition}
\newtheorem*{nsc}{Numerical scheme}
\newtheorem{xmp}[thm]{Example}
\newenvironment{sop}{\textit{Scatch of proof}}{\hfill$\Box$}
\newenvironment{proof2}[1]{\noindent \textit{Proof of #1 .}}{\hfill$\Box$}

\selectlanguage{english}

\begin{abstract}
A fully discrete Lagrangian scheme for solving a family of fourth order equations numerically is presented.
The discretization is based on the equation's underlying gradient flow structure w.r.t. the $L^2$-Wasserstein distance,
and adapts numerous of its most important structural properties by construction, as conservation of mass and entropy-dissipation. 

In this paper, the long-time behaviour of our discretization is analyzed: 
We show that discrete solutions decay exponentially to equilibrium at the same rate as smooth solutions of the origin problem.
Moreover, we give a proof of convergence of discrete entropy minimizers towards Barenblatt-profiles or Gaussians, respectively,
using $\Gamma$-convergence.
\end{abstract}

\maketitle

%
\section{Introduction}
%
%
In this paper, we propose and study a fully discrete numerical scheme for a family of nonlinear fourth order equations
of the type
\begin{align}
	\partial_t u = -\big(u (u^{\alpha-1}u_{xx}^\alpha)_x\big)_x + \lambda(xu)_x 
	 \quad\textnormal{for } x\in\Omega=\setR, \, t>0 \label{eq:fofo}
\end{align}
and $u(0,.) = u^0$ on $\Omega$ at initial time $t=0$. 
The initial density $u^0\geq0$ is assumed to be compactly supported and integrable with total mass $M>0$, 
and we further require strict positivity of $u^0$ on $\operatorname{supp}(u^0)=[a,b]$.
For the sake of simplicity, let us further assume that $M=1$.
We are especially interested in the long-time behaviour of discrete solutions and their rate of decay towards equilibrium.
For the exponent in \eqref{eq:fofo}, we consider values $\alpha\in[\frac{1}{2},1]$, and assume $\lambda\geq0$. 
The most famous examples for parabolic equations described by \eqref{eq:fofo} are the so-called
\emph{DLSS equation} for $\alpha=\frac{1}{2}$, 
(first analysed by Derrida, Lebowitz, Speer and Spohn in \cite{DLSS1,DLSS2}
with application in semi-cunductor physics)
and the \emph{thin-film equation} for $\alpha=1$ --- 
indeed, for other values of $\alpha$, references are very rare in the literature, except 
\cite{MMS} of Matthes, McCann and Savaré.

Due to the physically motivated origin of equation \eqref{eq:fofo} (especially for $\alpha=\frac{1}{2}$ and $\alpha=1$),
it is not surprising that solutions to \eqref{eq:fofo} carry many structural properties
as for instance nonnegativity, the conservation of mass and the dissipation of (several) entropy functionals.
In section \ref{sec:structure}, we are going to list more properties of solutions to \eqref{eq:fofo}.
For the numerical approximation of solutions to \eqref{eq:fofo}, 
it is hence natural to ask for structure-preserving discretizations that inherit at least some of those properties.
A minimum criteria for such a scheme should be the preservation of non-negativity,
which can already be a difficult task, if standard discretizations are used.
So far, many (semi-)discretizations have been proposed in the literature, 
and most of them keep some basic structural properties of the equation's underlying nature.
Take for example \cite{BEJnum,CJTnum,JPnum,JuVi}, where positivity appears as a conclusion
of Lyaponov functionals --- a logarithmic/power entropy \cite{BEJnum,CJTnum,JPnum}
or some variant of a (perturbed) information functional.
But there is only a little number of examples, where structural properties of equation \eqref{eq:fofo}
are adopted from the discretization by construction. 
A very first try in this direction was a fully Lagrangian discretization for the DLSS equation 
by Düring, Matthes and Pina \cite{DMMnum}, which is based on its $L^2$-Wasserstein gradient flow representation 
and thus preserves non-negativity and dissipation of the Fisher-information.
A similar approach was then applied \cite{dlssv3}, again for the special case $\alpha=\frac{1}{2}$,
where we even showed convergence of our numerical scheme, 
which was -- as far as we know -- the first convergence proof of a fully discrete numerical scheme for the DLSS equation,
which additionally dissipates \emph{two} Lyapunov functionals.

%
\subsection{Description of the numerical scheme}\label{sec:scheme}
We are now going to present a scheme, which is practical, stable and easy to implement.
In fact our dicretization seems to be so mundane that one would not assume any special properties therein, at the first glance. 
But we are going to show later in section \ref{sec:structure}, that our numerical approximation 
can be derived as a natural restriction of a $L^2$-Wasserstein gradient flow 
in the potential landscape of the so-called \emph{perturbed information functional} 
\begin{align}\label{eq:info}
	\Fal(u) = \frac{1}{2\alpha}\intom\big(\partial_x u^\alpha\big)^2\dd x + \frac{\lambda}{2}\intom |x|^2 u(x) \dd x,
\end{align}
into a discrete Lagrangian setting, thus preserves a deep structure. 
The starting point for our discretization is the \emph{Lagrangian representation} of \eqref{eq:fofo}.
Since each $u(t,\cdot)$ is of mass $M$, there is a Lagrangian map $\theX(t,\cdot):[0,M]\to\Omega$
--- the so-called \emph{pseudo-inverse distribution function} of $u(t,\cdot)$ ---
such that
\begin{align}
  \label{eq:pseudo}
  \xi = \int_{-\infty}^{\theX(t,\xi)}u(t,x)\dd x, \quad \text{for each $\xi\in[0,M]$}.
\end{align}
Written in terms of $\theX$, the Wasserstein gradient flow for $\Fal$ 
turns into an $L^2$-gradient flow for
\begin{align*}
  \Fal(u\circ\theX) = \frac{1}{2\alpha}\int_0^M \left[\frac{1}{\theX_\xi^\alpha}\right]_\xi^2\frac{1}{\theX_\xi}\dd\xi
											+\frac{\lambda}{2}\int_0^M \theX^2\dd\xi,
\end{align*}
that is,
\begin{align}
  \label{eq:zeq}
  \partial_t\theX = \frac{2\alpha}{(2\alpha+1)^2}\partial_\xi\big(Z^{\alpha+\frac{3}{2}}\partial_{\xi\xi}Z^{\alpha+\frac{1}{2}}\big) + \lambda\theX, 
  \quad \text{where} \quad Z(t,\xi):=\frac1{\partial_\xi\theX(t,\xi)}=u\big(t,\theX(t,\xi)\big). 
\end{align}
To build a bridge from \eqref{eq:zeq} to the origin equation \eqref{eq:fofo}, remember that \eqref{eq:fofo} can be written as a transport equation,
\begin{align}\label{eq:transport}
	\partial_t u + (u\vel)_x = 0, \quad\textnormal{with velocity field}\quad \vel = -\left(\frac{\delta\Fal(u)}{\delta u}\right)_x,
\end{align}
where $\delta\Fal(u)/\delta u$ denotes the Eulerian first variation. 
So take the time derivative in equation \eqref{eq:pseudo} and use \eqref{eq:transport}, then a formal calculation yields
\begin{align*}
	0 &= \partial_t\theX(t,\xi) u(t,\theX(t,\xi)) + \int_{-\infty}^{\theX(t,\xi)} \partial_t u(t,\theX(t,\xi))\dd x \\
		&= \partial_t\theX(t,\xi) u(t,\theX(t,\xi)) - \int_{-\infty}^{\theX(t,\xi)} (u\vel)_x(t,x)\dd x
		= \partial_t\theX(t,\xi) u(t,\theX(t,\xi)) - (u\vel)\circ\theX(t,\xi).
\end{align*}
This is equivalent to 
\begin{align*}
	\partial_t\theX(t,\xi) = \vel\circ\theX(t,\xi)\quad\textnormal{for } (t,\xi)\in(0,+\infty)\times[0,M],
\end{align*}
which is further equivalent to \eqref{eq:zeq}.

Before we come to the proper definition of the numerical scheme, 
we fix a spatio-temporal discretization parameter $\Delta=(\thet;\delta)$:
Given $\tau>0$, introduce varying time step sizes $\thet=(\tau_1,\tau_2,\ldots)$ with $\tau_n\in(0,\tau]$, 
then a time decomposition of $[0,+\infty)$ is defined by $\{t_n\}_{n\in\setN}$ with $t_n:=\sum_{j=1}^n \tau_n$.
As spatial discretization, fix $K\in\setN$ and $\delta=M/K$, and declare an equedistant decomposition of the mass space $[0,M]$ through
the set $\{\xi_k\}_{k=0}^K$ with $\xi_k:=k\delta$, $k=0,\ldots,K$.

Our numerical scheme is now defined as a standard discretization of equation \eqref{eq:zeq}:
\begin{nsc}
Fix a discretization parameter $\Delta=(\thet;\delta)$.
Then for any $(\alpha,\lambda)\in[\frac{1}{2},1]\times[0,+\infty)$
and any initial density function $u^0\in L^1(\Omega)$ satisfying the above requirements, 
a numerical scheme for \eqref{eq:fofo} is recursively given as follows:
\begin{enumerate}
\item For $n=0$, define an initial sequence of monotone values $\xvec_\Delta^0:=(x_0^0,\ldots,x_K^0)\in\setR^{K+1}$
uniquely by $x_0^0=a$, $x_K=b$ and
\begin{align*}
	\xi_k = \int_{x_{k-1}^0}^{x_k^0} u^0(x) \dd x,\quad\textnormal{for any } k=1,\ldots,K-1.
\end{align*}
The vector $\xvec_\Delta^0$ describes a non-equidistant decomposition of the support $[a,b]$ of the initial density function $u^0$. 
In any interval $[x_{k-1}^0,x_k^0]$, $k=1,\ldots,K$, the density $u^0$ has mass $\delta$.
\item For $n\geq1$, define recursively a monotone vector $\xvec_\Delta^n:=(x_0^n,\ldots,x_K^n)\in\setR^{K+1}$ as a solution of the system,
consisting of $(K+1)$-many equations
\begin{align}
  \label{eq:dgf}
  \frac{x^n_k-x^{n-1}_k}{\tau_n} 
  = \frac{2\alpha}{(2\alpha+1)^2\delta}\left[
    (z^n_\kph)^{\alpha+\frac{3}{2}}[\D^2\zvec^{\alpha+\frac{1}{2}}]_\kph
    -(z^n_\kmh)^{\alpha+\frac{3}{2}}[\D^2\zvec^{\alpha+\frac{1}{2}}]_\kmh
  \right]
	+ \lambda x_k,
\end{align}
with $k=0,\ldots,K$, where the values $z_{\ell-\frac{1}{2}}^n\geq 0$ are defined by
\begin{align}\label{eq:zvec}
	z_{\ell-\frac12}^n = \begin{cases}\frac{\delta}{x_\ell^n - x_{\ell-1}^n} &\textnormal{, for } \ell=1,\ldots,K \\
													0 &\textnormal{, else}
							\end{cases},
\end{align}
and
\begin{align*}
	[\D^2\zvec^{\alpha+\frac{1}{2}}]_\kmh:= \delta^{-2}\big(z_{\kph}^{\alpha+\frac{1}{2}}-2z_\kmh^{\alpha+\frac{1}{2}}+z_\kmd^{\alpha+\frac{1}{2}}\big).
\end{align*}
We later show in Proposition \ref{prp:wellposed}, that the solvability of the system \eqref{eq:dgf} is guaranteed.
\end{enumerate}
The above procedure $(1)-(2)$ yields a sequence of monotone vectors $\xvec_\Delta:=(\xvec_\Delta^0,\xvec_\Delta^1,\ldots,\xvec_\Delta^n,\ldots)$, 
and any entry $\xvec_\Delta^n$ defines a spatial decomposition of the compact interval $[x_0^n,x_K^n]\subset\Omega$, $n\in\setN$.
Fixing $k=1,\ldots,K$, the sequence $n\mapsto x_k^n$ defines a discrete temporal evolution of spatial grid points in $\Omega$,
and if one assigns each interval $[x_{k-1}^n,x_k^n]$ a constant mass package $\delta$,
the map $n\mapsto [x_{k-1}^n,x_k^n]$ characterizes the temporal movement of mass.
Hence $\xvec_\Delta$ is uniquely related to a sequence of local constant density functions $u_\Delta:=(u_\Delta^0,u_\Delta^1,\ldots,u_\Delta^n,\ldots)$, 
where each function $u_\Delta^n:\Omega\to\setR_+$ holds
\begin{align}\label{eq:cf}
  u_\Delta^n (x) = \cf_\theh[\xvec_\Delta^n] := \sum_{k=1}^K\frac\delta{x^n_k-x^n_{k-1}} \indy_{(x_{k-1},x_k]}(x).
\end{align}
\end{nsc}

We will see later in section \ref{sec:structure_cont}, 
that the information functional $\Fal$ can be derived using the dissipation of the entropy
\begin{equation*}
	\Hal(u) = \intom \phia(u)\dd x + \frac{\La}{2}\intom |x|^2u(x)\dd x , \quad\textnormal{with}\quad
	\phia(s):=\begin{cases}  \Theta_\alpha \frac{s^{\alpha+1/2}}{\alpha-1/2}, & \alpha\in(\tfrac{1}{2},1] \\
													 \Theta_{1/2} s\ln(s), &\alpha=\tfrac{1}{2} \end{cases},
\end{equation*}
with constants $\Theta_\alpha := \sqrt{2\alpha}/(2\alpha+1)$, and $\La := \sqrt{\lambda/(2\alpha+1)}$.
As replacements for the entropy $\Hal$ and the perturbed information functional $\Fal$, we introduce
\begin{align}\label{eq:Halz}
  \Halz(\xvec) :=\delta\sum_{k=1}^K \fa(z_\kmh) + \frac{\La}{2}\delta\sum_{k=0}^K |x_k|^2,
	\quad\textnormal{with}\quad
	\fa(s) := \begin{cases} \Theta_\alpha \frac{s^{\alpha-1/2}}{\alpha-1/2}, & \alpha\in(\tfrac{1}{2},1] \\
													\Theta_{1/2} \ln(s) ,& \alpha=\tfrac{1}{2}\end{cases},
\end{align}
and
\begin{align}\label{eq:Falz}
	\Falz(\xvec) := \Theta_\alpha^2\delta\sum_{k\in\ival}\left(\frac{z_\kph^{\alpha+\frac{1}{2}}-z_\kmh^{\alpha+\frac{1}{2}}}{\delta}\right)^2	+ \frac{\lambda}{2}\delta\sum_{k=0}^K |x_k|^2.
\end{align}
%
\subsection{Familiar schemes}
The construction of numerical schemes as a solution of discrete Wasserstein gradient flows
with Lagrangian representation is not new in the literature.
Many approaches in this spirit have been realised for second-order diffusion equation \cite{Budd,BCW,MacCamy,Russo},
but also for chemotaxis systems \cite{BCC},
for non-local aggregation equations \cite{CarM,Mary},
and for variants of the Boltzmann equation \cite{GosT2}. 
We further refer \cite{Kinderlehrer} to the reader interested in a very general numerical treatement of Wasserstein gradient flows.
In case of fourth order equations, there are some results for the thin-film equation and its more general 
version, the Hele-Shaw flow, see \cite{Naldi,GosT2}, but converegence results are missing.
Rigorous stability and convergence results for \emph{fully discrete} schemes are rare and can just be found in \cite{GosT,dde} for second order equations,
and in \cite{dlssv3} for the DLSS equation.
However, there are results available for \emph{semi-discrete} Lagrangian approximations,
see e.g. \cite{ALS,Evans}.

%
\subsection{Main results}
In this section, fix a discretization $\Delta=(\thet;\delta)$ with $\tau,\delta>0$.
For any solution $\xvec_\Delta$ of \eqref{eq:dmm2},
we will further denote by $u_\Delta=(u_\Delta^0,u_\Delta^1,\ldots)$ the corresponing sequence of local constant density functions, as defined in \eqref{eq:cf}.

All analytical results that will follow, arise from the very fundamental observation, that
solutions to the scheme defined in section \ref{sec:scheme} can be successively derived 
as minimizers of the \emph{discrete minimizing movemend scheme}
\begin{align}\label{eq:dmm2}
	\xvec\mapsto\frac\delta{2\tau_n}\sum_k \big(x_k-x_k^{n-1})^2 + \Falz(\xvec).
\end{align}
An immediate consequence of the minimization procedure is, that solutions $\xvec_\Delta^n$ dissipate the functional $\Falz$.

Concerning the long-time behaviour of solutions $\xvec_\Delta$, remarkable similarities to the continuous case appear.
Assuming first the case $\lambda>0$, it turns out that
the unique minimizer $\xvecm$ of $\Halz$ is even a minimizer of the discrete information functional $\Falz$,
and the corresponding set of density functions $\um = \cf_\theh[\xvecm]$
converges for $\delta\to0$ towards a Barenblatt-profile $\bal$ or Gaussian $\bhl$, respectively, that is defined by
\begin{gather}
	\bal = \big(a-b|x|^2\big)_+^{1/(\alpha-1/2)},\quad b = \frac{\alpha-1/2}{\sqrt{2\alpha}}\La 
	\quad\textnormal{if } \alpha> 1/2 \textnormal{ and} \label{eq:Barenblatt}\\
	\bhl = a e^{-\Lh|x|^2} \quad\textnormal{if } \alpha=1/2, \label{eq:Gaussian}
\end{gather}
where $a\in\setR$ is chosen to conserve unit mass.
Beyond this, solutions $\xvec_\Delta^n$ satisfying \eqref{eq:dgf}
converge as $n\to\infty$ towards a minimizer $\xvecm$ of $\Falz$
with an exponential decay rate, which is "asymptotically equal" to the one obtained in the continuous case.
The above results are merged in the following theorems:
\begin{thm}\label{thm:main1}
Assume $\lambda>0$.
Then the sequence of minimizers $\um$ holds
\begin{align}
	\um \xrightarrow{\delta\to0} \bal &,\textnormal{ strongly in } L^p(\Omega) 
	\textnormal{ for any } p\geq 1, \label{eq:Lconv} \\
	\umh \xrightarrow{\delta\to0} \bal &,\textnormal{ uniformly on } \Omega, \label{eq:uniform}
\end{align}
where $\umh$ is a locally affine interpolation of $\um$ defined in Lemma \ref{lem:Lconv}.
\end{thm}
\begin{thm}\label{thm:main2}
For $\lambda>0$, any sequence of monotone vectors $\xvec_\Delta$ satisfying \eqref{eq:dmm2} dissipates the entropies $\Halz$ and $\Falz$ at least exponential, i.e.
\begin{align}
	\Halz(\xvec_\Delta^n) - \Halzm &\leq \left(\Halz(\xvec_\Delta^0) - \Halzm\right) e^{-\frac{2\lambda}{1+\lambda\tau} t_n}, 
	\quad\textnormal{and} \label{eq:expH} \\
	\Falz(\xvec_\Delta^n) - \Falzm &\leq \left(\Falz(\xvec_\Delta^0) - \Falzm\right) e^{-\frac{2\lambda}{1+\lambda\tau} t_n}, \label{eq:expF}
\end{align}
with $\Halzm=\Halz(\xvecm)$ and $\Falzm=\Falz(\xvecm)$. The associated sequence of densities $u_\Delta$ further holds
\begin{align}\label{eq:CK}
	\|u_\Delta^n - \um\|_{L^1(\Omega)}^2 \leq c_{\alpha,\lambda} \left(\Halz(\xvec_\Delta^0) - \Halzm\right) e^{-\frac{2\lambda}{1+\lambda\tau} t_n},
\end{align}
for any time step $n=1,2,\ldots$, where $c_{\alpha,\lambda}>0$ depends only on $\alpha,\lambda$. 
\end{thm}
Let us now consider the zero-confinement case $\lambda=0$.
In the continuous setting, the long-time behaviour of solutions to \eqref{eq:fofo} with $\lambda=0$
can be studied by a rescaling of solutions to \eqref{eq:fofo} with $\lambda>0$. 
We are able to translate this methode into the discrete case and derive a discrete counterpart of \cite[Corollary 5.5]{MMS},
which describes the intermediate asymptotics of solutions that approach self-similar Barenblatt profiles as $t\to\infty$.
\begin{thm}\label{thm:main3}
Assume $\lambda=0$ and take a sequence of monotone $\xvec_\Delta^n$ satisfying \eqref{eq:dmm2}.
Then there exists a constand $c_\alpha>0$ depending only on $\alpha$, such that
\begin{align*}
	\|u_\Delta^n - \bantn\|_{L^1(\Omega)}
	\leq c_\alpha\sqrt{\Hanz(\xvec_\Delta^0) - \Hanzm} (R_\Delta^n)^{-1},
	\quad\textnormal{with}\quad R_\Delta^n := \big(1+a_\tau(2\alpha+3)t_n\big)^{\frac{1}{b_\tau(2\alpha+3)}},
\end{align*}
where $\bantn$ is a rescaled discrete Barenblatt profile and $a_\tau,b_\tau>0$, 
such that $a_\tau,b_\tau\to 1$ for $\tau\to0$, see section \ref{sec:confn} for more details.
\end{thm}

Before we come to the analytical part of this paper, we want to point out the following:
The ideas for the proofs of Theorem \ref{thm:main2} and \ref{thm:main3} 
are mainly guided by the techniques developd in \cite{MMS}. 
The remarkable observation of this work is the fascinating structure preservation of our discretization,
which allows us to adapt nearly any calculation from the continuous theory for the discrete setting.

%
\subsection{Structure of paper}
In the following section \ref{sec:structure}, 
we point out some of the main structural features of equation \eqref{eq:fofo} and the functionals $\Hal$ and $\Fal$,
and show that our scheme rises from a discrete $L^2$-Wasserstein gradient flow,
so that many properties of the continuous flow are inherited.
Section \ref{sec:equi} treats the analysis of discrete equilibria in case of positive confinement $\lambda>0$: 
we prove convergence of discrete stationary states to 
Barenblatt-profiles or Gaussians, respectivelly, 
and analyse the asymptotics of discrete solutions for $\lambda=0$.
Finally, some numerical experiments are presented in section \ref{sec:num}.

%
\section{Structural properties --- continuous vs. discrete case}\label{sec:structure}
%

%
\subsection{Structural properties of equation \eqref{eq:fofo}}\label{sec:structure_cont}
The family of fourth order equations \eqref{eq:fofo}
carries a bunch of remarkable structural properies.
The most fundamental one is the conservation of mass, i.e. 
$t\mapsto \|u(t,\cdot)\|_{L^1(\Omega)}$ is a constant function for $t\in[0,+\infty)$ and attains the value $M:=\|u^0\|_{L^1(\Omega)}$.
This is a naturally given property, if one interprets solutions to \eqref{eq:fofo} as gradient flows
in the potential landscape of the perturbed information functional
\begin{align}
	\Fal(u) = \frac{1}{2\alpha}\intom\big(\partial_x u^\alpha\big)^2\dd x + \frac{\lambda}{2}\intom |x|^2 u(x) \dd x,
\end{align}
equipped with the $L^2$-Wasserstein metric $\wass$. 
As an immediate consequence, $\Fal$ is a Lyapunov functional, 
and one can find infinitely many other (formal) Lyapunov functionals 
at least for special choices of $\alpha$ --- 
see \cite{BLS,CCTdlss,JMalg} for $\alpha=\frac{1}{2}$ or \cite{BGruen,CaTothin,GiOt} for $\alpha=1$.
Apart from $\Fal$, one of the most important such Lyapunov functionals is given by
the $\La$-convex entropy
\begin{equation}
	\label{eq:entropy}
	\Hal(u) = \intom \phia(u)\dd x + \frac{\La}{2} \intom |x|^2 u(x) \dd x, \quad
	\phia(s):=\begin{cases}  \Theta_\alpha \frac{s^{\alpha+1/2}}{\alpha-1/2}, & \alpha\in(\tfrac{1}{2},1] \\
													 \Theta_{1/2} s\ln(s), &\alpha=\tfrac{1}{2} \end{cases}.
\end{equation}
It turns out that the functionals $\Fal$ and $\Hal$ are not just Lyapunov functionals,
but share numerous remarkable similiarities. One can indeed see \eqref{eq:fofo} 
as an higher order extension of the second order \emph{porous media/heat equation} \cite{JKO} 
\begin{align}
	\label{eq:heat}
	\partial_s v = -\thegrad\Hal(v) = -\Theta_\alpha\partial_{xx}(v^\alpha) + \La(xu)_x,
\end{align}
which is nothing less than the $L^2$-Wasserstein gradient flow of $\Hal$. 
Furthermore, the unperturbed functional $\Fan$, i.e. $\lambda=\La=0$, equals the dissipation of $\Han$ along its own gradient flow,
\begin{align}
	\label{eq:magic}
	\Fan(v(s)) = -\frac{\dd}{\dd s}\Han(v(s)).
\end{align}
In view of of the gradient flow structure, 
this relation makes equation \eqref{eq:fofo} the ``big brother'' of the porous media/heat equation \eqref{eq:heat}, 
see \cite{DMfourth,MMS} for structural consequences.
Another astonishing common feature is the correlation of $\Fal$ and $\Hal$ by 
the so-called \emph{fundamental entropy-information relation}: For any $u\in\dens$ with $\Hal(u)<\infty$, it holds
\begin{align}\label{eq:feir}
	\Fal(u) = |\thegrad\Hal|^2 + (2\alpha-1)\La\Hal(u),
	\quad\textnormal{for any } \lambda\geq 0,
\end{align}
see \cite[Corollary 2.3]{MMS}. This equation is a crucial tool for the analysis of equilibria of both functionals 
and the corresponding long-time behaviour of solutions to \eqref{eq:fofo} and \eqref{eq:heat}.

In addition to the above listing, a typical property of diffusion processes like \eqref{eq:fofo} or \eqref{eq:heat} with positive confinement $\lambda,\La>0$ is
the convergence towards unique stationary solutions $\uinf$ and $\vinf$, respectively, independent of the choice of initial data.
It is maybe one of the most surprising facts, that both equations \eqref{eq:fofo} and \eqref{eq:heat} share the same steady state, i.e. 
the stationary solutions $\uinf$ and $\vinf$ are identical. 
Those stationary states are solutions of the elliptic equations
\begin{align}\label{eq:ell}
	-\big(\prss(u)\big)_{xx} + \La (x u)_x = 0,
\end{align}
with $\prss(s):=\Theta_\alpha s^{\alpha+1/2}$, and have the form of Barenblatt profils or Gaussians, respectively,
see definition \eqref{eq:Barenblatt}\&\eqref{eq:Gaussian}.
This was first observed by Denzler and McCann in \cite{DMfourth}, and further studied in \cite{MMS} using the Wasserstein gradient flow structure of both
equations and their remarkable relation via \eqref{eq:magic}.

In case of $\alpha\in\{\frac{1}{2},1\}$, the mathematical literature is full of numerous results,
which is because of the physical importance of \eqref{eq:fofo} in those limiting cases. 

\subsubsection{DLSS euqation}
As already mentioned at the very beginning, 
the DLSS equation --- \eqref{eq:fofo} with $\alpha=\frac{1}{2}$ --- rises from the Toom model \cite{DLSS1,DLSS2} in one spatial dimension on 
the half-line $[0,+\infty)$, and was used to describe interface fluctuations, therein.
Moreover, the DLSS equation also finds application in semi-conductor physics, 
namely as a simplified model (low-temperature, field-free) for a quantum drift diffusion system for electron densities, see \cite{JPdlss}.

From the analytical point of view, a big variety of results in different settings has been developed in the last view decades.
For results on existence and uniqueness, we refer f.i. \cite{BLS,Funique,GJT,GST,JMdlss,JPdlss},
and \cite{CCTdlss,CTthin,CDGJ,GST,JMdlss,JTdecay,MMS} for qualitative and quantitative descriptions of the long-time behaviour.
The main reason, which makes the research on this topic so non-trivial, is a lack of comparison/maximum principles as 
in the theory of second order equations \eqref{eq:heat}.
And, unfortunatelly, the abscents of such analytical tools is not neglectable, as the work \cite{BLS} of Bleher et.al shows:
as soon as a solution $u$ of \eqref{eq:fofo} with $\alpha=\frac{1}{2}$ is strictly positive, one can show that it is even $C^\infty$-smooth,
but there are no regularity results available from the moment when $u$ touches zero.
The problem of strictly positivity of such solutions seems to be a difficult task, since it is still open. 
This is why alternative theories for non-negative weak solutions 
have more and more become matters of great interest, as f.i. an approach based on entropy methodes developed in \cite{GST,JMdlss}.

\subsubsection{Thin-film equation}
The thin-film equation --- \eqref{eq:fofo} with $\alpha=1$ --- is of similar physically importance as the DLSS equation,
since it gives a dimension-reduced description of the free-surface problem with the Navier-Stokes equation in the case of laminar flow, \cite{Oron}.
In case of linear mobility --- which is exactly the case in our situation --- the thin-film equation
can also be used to describe the pinching of thin necks in a \emph{Hele-Shaw} cell in one spatial dimension,
and thus plays an extraordinary role in physical applications. 
To this topic, the literature provides some interesting results in the framework of entropy methods, 
see \cite{Ulusoy,CaTothin,GiOt}.
In the (more generel) case of non-negative mobility functions $m$, i.e.
\begin{align}\label{eq:thinfilm}
	\partial_t u = -\operatorname{div}(m(u)\operatorname{D}\Delta u),
\end{align}
one of the first achievements to this topic available in the mathematical literature was done by Bernis and Friedman \cite{BernisF}.
The same equation is observed in \cite{Bertsch}, treating a vast number of results to numerous mobility functions of physical meaning.
There are several other references in this direction, f.i. Grün et. al \cite{BGruen,Passo,Gruen},
concerning long-time behaviour of solutions and the non-trivial question of spreading behaviour of the support.

%
\subsection{Structure-preservation of the numerical scheme}\label{sec:structure_discrete}
In this section, we try to get a better intuition of the scheme in section \ref{sec:scheme}.
Foremost we will derive \eqref{eq:dgf} as a discrete system of Euler-Lagrange equations
of a variational problem that rises from a $L^2$-Wasserstein gradient flow restricted on a 
discrete submanifold $\densN$ of the space of probability measures $\dens$ on $\Omega$.
This is why the numerical scheme holds several discrete analogues of the results discussed in the previous section.
As the following section shows, some of the inherited properties are obtained by construction (f.i. preservation of mass and dissipation of the entropy),
where others are caused by the underlying dicsrete gradient flow structure and a smart choice of 
a discrete $L^2$-Wasserstein distance.
Moreover it is possible to prove that the entropy and the information functional share the same minimizer $\xvecm$ even in the discrete case, 
and solutions of the discrete gradient flow converges with an exponential rate to this stationary state. 
The prove of this observation is more sophisticated, 
that is why we dedicate an own section (section \ref{sec:equi}) to the treatment of this special property.

\subsubsection{Ansatz space and discrete entropy/information functionals}
The entropies $\Hal$ and $\Fal$ as defined in \eqref{eq:entropy}\&\eqref{eq:info} are non-negative functionals on $\dens$.
If we first consider the zero-confinement case $\lambda=0$,
one can derive in analogy to \cite{dde} the discretization in \eqref{eq:Halz} of $\Han$ just by restriction 
to a finite-dimensional submanifold $\densN$ of $\dens$:
For fixed $K\in\setN$, the set $\densN$ consists of all local constant density functions
$u=\cf_\theh[\xvec]$ (remember definition \eqref{eq:cf}), such that $\xvec\in\setR^{K+1}$ is a monotone vector, i.e.
\begin{align*}
  \xvec\in\xseqN := \big\{ (x_0,\ldots,x_K) \,\big|\, x_0 < x_1 < \ldots < x_{K-1} < x_K \big\} \subseteq \setR^{K+1}.
\end{align*}
Such density functions $u=\cf_\theh[\xvec]\in\densN$ bear a one-to-one relation to their \emph{Lagrangians} or \emph{Lagrangian maps},
which are defefined on the mass grid $[0,M]$ with uniform decomposition $(0=\xi_0,\ldots,\xi_k,\ldots,\xi_K=M)$.
More precicely, we define for $\xvec\in\xseqN$ the local affine and monotonically increasing function $\cX=\cX_\theh[\xvec]:[0,M]\to\Omega$,
such that $\cX(\xi_k)=x_k$ for any $k=0,\ldots,K$. It then holds $u\circ\cX=\frac{1}{\cX_\xi}$ for $u\in\densN$ and its corresponding 
Lagrangian map.
For later analysis, we introduce in addition to the decomposition $\{\xi_k\}_{k=0}^K$ the intermediate values 
$(\xi_{k-\frac{1}{2}},\xi_{\frac{3}{2}},\ldots,\xi_{K-\frac{1}{2}})$ by $\xi_\kmh=\frac{1}{2}(\xi_k+\xi_{k-1})$ for $k=1,\ldots,K$.

In view of the entropy's discretization, this implies using \eqref{eq:Halz}\&\eqref{eq:entropy}, a change of variables $x=\cX_\theh[\xvec]$, 
and the definition \eqref{eq:zvec} of the $\xvec$-dependent vectors $\zvec$
\begin{align*}
  \Hanz(\xvec) = \Han(\cf_\theh[\xvec]) 
  = \int_\Omega \phia\big(\cf_\theh[\xvec]\big)\dd x
  = \delta\sum_{k=1}^K \fa(z_\kmh),
\end{align*}
which is perfectly compatible with \eqref{eq:Halz}.
Obviously, one cannot derive the discrete information functional $\Fanz$ in the same way, since $\Fan$ is not defined on $\densN$.
So instead of restriction, we mimic property \eqref{eq:magic} that is for any $\xvec\in\xseqN$
\begin{align*}
	\Fanz(\xvec) = \delta^{-1}\grad\Hanz(\xvec)^T\grad\Hanz(\xvec) = \spr{\wgrad\Hanz(\xvec)}{\wgrad\Hanz(\xvec)}.
\end{align*}
Here, the $k$th component of $\grad f(\xvec)$ holds $[\grad f(\xvec)]_k=\partial_{x_k} f(\xvec)$ for any $k=0,\ldots,K$ and arbitrary function $f:\xseqN\to\setR$. 
Moreover, we set $\wgrad f(\xvec) = \delta^{-1}\grad f(\xvec)$ and introduce for $\vvec,\wvec\in\setR^{K+1}$ 
the scalar product $\spr{\cdot}{\cdot}$ by
\begin{align*}
  \spr{\vvec}{\wvec} = \delta\sum_{k=0}^{K}v_kw_k, 
	\quad\textnormal{with induced norm}\quad \nrm{\vvec} = \sqrt{\spr{\vvec}{\vvec}}.
\end{align*}
\begin{xmp}
Each component $z_\kappa$ of $\zvec=\cz_\theh[\xvec]$ is a function on $\xseqN$, 
and
\begin{align}
	\label{eq:zrule}
	\grad z_\kappa = -z_\kappa^2\,\frac{\ee_{\kappa+\frac12}-\ee_{\kappa-\frac12}}{\delta},
\end{align}
where we denote for $k=0,\ldots,K$ by $\ee_k\in\setR^{K+1}$ the $(k+1)$th canonical unit vector.
\end{xmp}
\begin{rmk}
One of the most fundamental properties of the $L^2$-Wasserstein metric $\wass$ on $\dens$ in one space dimension is 
its excplicit representation in terms of Lagrangian coordinates. 
We refer \cite{AGS,VilBook} for a comprehensive introduction to the topic.
This enables to prove the existence of $K$-independent constants $c_1,c_2>0$, such that
\begin{align}\label{eq:metricequivalent}
	c_1\nrm{\xvec-\yvec} \leq \wass(\cf_\theh[\xvec],\cf_\theh[\yvec]) \leq c_2\nrm{\xvec-\yvec},\quad\textnormal{for all } \xvec,\yvec\in\xseqN.
\end{align}
A proof of this statement for $\Omega=[a,b]\subset(-\infty,+\infty)$ is given in \cite[Lemma 7]{dde},
and can be easily recomposed for $\Omega=\setR$.
\end{rmk}
Let us further introduce the sets of (semi)-indizes 
\begin{align*}
  \ival = \{0,1,\ldots,K\},\quad\text{and}\quad
  \hval = \Big\{\frac12,\frac32,\ldots,K-\frac12\Big\}.
\end{align*}
The calculation \eqref{eq:zrule} in the above example yields the expizit representation of the gradient $\grad\Halz(\xvec)$,
\begin{equation}\begin{split}\label{eq:gradH}
	\grad\Hanz(\xvec) = \Theta_\alpha\delta\sum_{\kappa\in\hval} z_\kappa^{\alpha+\frac{1}{2}} \frac{\ee_{\kappm}-\ee_{\kappp}}\delta ,
\end{split}\end{equation}
and further of the discretized information functional
\begin{align*}
	\Fanz(\xvec) = \nrm{\wgrad\Hanz(\xvec)}^2
	= \Theta_\alpha^2\delta\sum_{k\in\ival}\left(\frac{z_\kph^{\alpha+\frac{1}{2}}-z_\kmh^{\alpha+\frac{1}{2}}}{\delta}\right)^2.
\end{align*}

In the case of positive confinemend $\lambda>0$, we note that the drift potential $u\mapsto\intom|x|^2u(x)\dd x$ 
holds an equivalent representation in terms of Lagrangian coordinates, that is namely $\cX\mapsto\int_0^M |\cX(\xi)|^2\dd\xi$.
In our setting, the simplest discretization of this functional is hence by summing-up over all values $x_k$ weighted with $\delta$.
This yields
\begin{align*}
	\Halz(\xvec) = \Hanz(\xvec) + \frac{\La}{2}\delta\sum_{k\in\ival}|x_k|^2,
	\quad\textnormal{and}\quad
	\Falz(\xvec) = \Fanz(\xvec) + \frac{\lambda}{2}\delta\sum_{k\in\ival}|x_k|^2
\end{align*}
as an extension to the case of positive $\lambda$, which is nothing else than \eqref{eq:Halz}\&\eqref{eq:Falz}.
Note in addition, that $\delta\sum_{k\in\ival}|x_k|^2 = \nrm{\xvec}^2$.

A first structural property of the above simple discretization is
convecity retention from the continuous to the discrete setting:
\begin{lem}\label{lem:Lconvex}
The functional $\xvec\mapsto\Halz$ is $\La$-convex, i.e.
\begin{align}\label{eq:Lconvex}
	\Halz\big((1-s)\xvec + s\yvec\big) \leq (1-s)\Halz(\xvec) + s\Halz(\yvec) - \frac{\La}{2}(1-s)s \nrm{\xvec-\yvec}^2
\end{align}
for any $\xvec,\yvec\in\xseqN$ and $s\in(0,1)$. It therefore admits a unique minimizer $\xvecm\in\xseqN$.
If we further assume $\La>0$, then it holds for any $\xvec\in\xseqN$
\begin{align}\label{eq:HleqF}
	\frac{\La}{2}\nrm{\xvec-\xvecm}^2 \leq \Halz(\xvec) - \Halz(\xvecm) \leq \frac{1}{2\La}\nrm{\wgrad\Halz(\xvec)}^2.
\end{align}
\end{lem}
\begin{proof}
If we prove \eqref{eq:Lconvex}, then the existence of the unique minimizer is a consequence \cite[Proposition 10]{dde}. 
By definition and a change of variables, we get for $\alpha\in(\frac{1}{2},1]$
\begin{align*}
  \Hanz(\xvec) &= \Han(\cf_\theh[\xvec]) 
	=\int_0^M \psia\left(\cX_\theh[\xvec]_\xi\right) \dd\xi ,\textnormal{ with }
	\psia(s) = \begin{cases} \Theta_\alpha \frac{s^{1/2-\alpha}}{\alpha-1/2}, & \alpha\in(\tfrac{1}{2},1] \\
													 -\Theta_{1/2}\ln(s), & \alpha=\tfrac{1}{2} \end{cases},
\end{align*}
hence $\xvec\mapsto\Hanz(\xvec)$ is convex. Since the functional $\xvec\mapsto\nrm{\xvec}^2$ holds trivially
\begin{align*}
	\nrm{(1-s)\xvec+s\yvec}^2 
	\leq (1-s)\nrm{\xvec}^2 + s\nrm{\yvec}^2 - (1-s)s\nrm{\xvec-\yvec}^2
\end{align*}
for any $\xvec,\yvec\in\xseqN$ and $s\in(0,1)$, the functionals $\Halz(\xvec) = \Hanz(\xvec) + \frac{\La}{2}\nrm{\xvec}^2$ hold \eqref{eq:Lconvex}.

Deviding \eqref{eq:Lconvex} by $s>0$ and passing to the limit as $s\downarrow0$ yields
\begin{align*}
	\Halz(\xvec) - \Halz(\yvec) \leq \grad\Halz(\xvec)(\xvec-\yvec) - \frac{\La}{2}\nrm{\xvec-\yvec}^2.
\end{align*}
The second inequality of \eqref{eq:HleqF} easily follows from Young's inequality $|ab|\leq \eps|a|^2+(2\eps)^{-1}\frac{1}{2}|b|^2$ with $\eps=(2\delta\La)^{-1}$,
and even holds for arbitrary $\yvec\in\xseqN$.

To get the first inequaltiy of \eqref{eq:HleqF}, we set $\xvec=\xvecm$ and again devide \eqref{eq:Lconvex} by $s>0$, then
\begin{align*}
	\frac{\Halz\big((1-s)\xvecm + s\yvec\big) - \Halz(\xvecm)}{s} \leq \Halz(\yvec) - \Halz(\xvecm) - \frac{\La}{2}(1-s) \nrm{\xvecm-\yvec}^2,
\end{align*}
where the left hand side is obviously non-negative for any $s>0$. Since $s>0$ was arbitrary, the statement is proven.
\end{proof}

As a further conclusion of our natural discretization, we get a \emph{discrete fundamental entropy-information relation} analogously 
to the continuous case \eqref{eq:feir}.
\begin{cor}\label{cor:dfeir}
For any $\lambda\geq0$, every $\xvec\in\xseqN$ with $\Hanz(\xvec)<\infty$ we have
\begin{align}\label{eq:dfeir}
	\Falz(\xvec) &= \nrm{\wgrad\Halz(\xvec)}^2 + (2\alpha-1)\La\Halz(\xvec), \quad\textnormal{ for } \alpha\in(\tfrac{1}{2},1] \quad\textnormal{and}\\
	\Fhlz(\xvec) &= \nrm{\wgrad\Hhlz(\xvec)}^2 + \Lh , \quad\textnormal{ for } \alpha=\tfrac{1}{2}
\end{align}
\end{cor}
\begin{rmk}
Note that the above seemingly appearing discontinuity at $\alpha=\frac{1}{2}$ is not real. 
For $\alpha>\frac{1}{2}$, the second term in right hand side of \eqref{eq:dfeir}
is explicitly given by
\begin{align*}
	(2\alpha-1)\La\Halz(\xvec) 
	&= (2\alpha-1)\La\left(\Theta_\alpha\delta\sum_{\kappa\in\hval}\frac{z_\kappa^{\alpha-1/2}}{\alpha-1/2} + \frac{\La}{2}\nrm{\xvec}^2\right) \\
	&= 2\La\Theta_\alpha\delta\sum_{\kappa\in\hval}z_\kappa^{\alpha-1/2} + (2\alpha-1)\frac{\La}{2}\nrm{\xvec}^2,
\end{align*}
For $\alpha\downarrow\frac{1}{2}$, one gets $\La\to\Lh$, $\Theta_\alpha\to\frac{1}{2}$ and especially $\delta\sum_{\kappa\in\hval}z_\kappa^{\alpha-1/2}\to M = 1$,
The drift-term vanishes since $(2\alpha-1)\to0$.
\end{rmk}
\begin{proof2}{Corollary \ref{cor:dfeir}}
Let us first assume $\alpha\in(\frac{1}{2},1]$.
A straight-forward calculation using the definition of $\nrm{.}$, $\wgrad$ and $\grad\Halz$ in \eqref{eq:gradH} yields 
\begin{equation}\begin{split}\label{eq:nrmHalz}
  \nrm{\wgrad\Halz(\xvec)}^2  &= \delta^{-1}\left\langle\grad\Halz(\xvec),\grad\Halz(\xvec)\right\rangle\\
	&= \nrm{\wgrad\Hanz(\xvec)}^2
			- 2\Theta_\alpha\La\delta\sum_{\kappa\in\hval} z_\kappa^{\alpha-\frac{1}{2}}
			+ \La^2\delta\sum_{k\in\ival} |x_k|^2.
\end{split}\end{equation}
Here we used the explicit representation of $\grad\Halz(\xvec)$, remember \eqref{eq:gradH},
\begin{align*}
		\grad\Halz(\xvec) = \Theta_\alpha\delta\sum_{\kappa\in\hval} z_\kappa^{\alpha+\frac{1}{2}} \frac{\ee_{\kappm}-\ee_{\kappp}}\delta
				+ \La\delta\sum_{k\in\ival}x_k\ee_k,
\end{align*}
and especially the definition of \eqref{eq:zvec}, which yields
\begin{align*}
	\delta^{-1}\left\langle \Theta_\alpha\delta\sum_{\kappa\in\hval} z_\kappa^{\alpha+\frac{1}{2}} \frac{\ee_{\kappm}-\ee_{\kappp}}\delta 
	,\La\delta\sum_{k\in\ival}x_k\ee_k\right\rangle
	&= \Theta_\alpha\La\delta\sum_{\kappa\in\hval}z_\kappa^{\alpha+\frac{1}{2}} \frac{x_\kappm-x_\kappp}\delta \\
	&= -\Theta_\alpha\La\delta\sum_{\kappa\in\hval}z_\kappa^{\alpha-\frac{1}{2}}
\end{align*}
Since $\alpha\neq\frac{1}{2}$, we can write 
$2\Theta_\alpha=(2\alpha-1)\frac{\Theta_\alpha}{\alpha-1/2}$.
Further note that the relation $\La=\sqrt{\lambda/(2\alpha+1)}$ yields
\begin{align*}
	\La^2 = \frac{\lambda}{2\alpha+1} = \frac{\lambda}{2}\left(\frac{1}{\alpha+1/2}\right)
				= \frac{\lambda}{2}\left(1 - \frac{\alpha-1/2}{\alpha+1/2}\right)
				= \frac{\lambda}{2}\left(1 - \frac{2\alpha-1}{2\alpha+1}\right).
\end{align*}
Using this information and the definition of $\Han$, we proceed in the above calculations by
\begin{align*}
	\nrm{\wgrad\Halz(\xvec)}^2  
	&= \Fanz(\xvec) - (2\alpha-1)\La\Hanz(\xvec) + \frac{\lambda}{2}\left(1 - \frac{2\alpha-1}{2\alpha+1}\right)\delta\sum_{k\in\ival} |x_k|^2 \\
	&= \Fanz(\xvec) - (2\alpha-1)\La\Hanz(\xvec) + \frac{\lambda}{2}\delta\sum_{k\in\ival} |x_k|^2 
																- (2\alpha-1)\frac{\La^2}{2}\delta\sum_{k\in\ival} |x_k|^2 \\
	&= \Falz(\xvec) - (2\alpha-1)\La\Halz(\xvec).
\end{align*}
In case of $\alpha=\frac{1}{2}$, we see that $\Theta_{1/2} = \frac{1}{2}$, and $\Lh=\sqrt{\lambda/2}$. We hence conclude in \eqref{eq:nrmHalz}
\begin{align*}
	\nrm{\wgrad\Hhlz(\xvec)}^2  
	=\nrm{\wgrad\Hhnz(\xvec)}^2
			- \Lh\delta\sum_{\kappa\in\hval} z_\kappa^0
			+ \frac{\lambda}{2}\delta\sum_{k\in\ival} |x_k|^2
	= \Falz(\xvec) - \Lh
\end{align*}
\end{proof2}

For the following reason, the above representation of $\Falz$ is indeed a little miracle:
From a naive point of view, one would ideally hope to gain a discrete counterpart of the 
fundamental entropy-information relation \eqref{eq:feir}, 
if one takes the one-to-one discretization of the $L^2$-Wasserstein metric, 
which is (in the language of Lagrangian vectors) realized by the norm 
$\xvec\mapsto \wass(\cf_\theh[\xvec],\cf_\theh[\xvec])$ instead of our simpler choice $\xvec\mapsto\nrm{\xvec}$. 
Indeed, with this ansatz, the above proof would fail
in the moment in which one tries to calculate the scalar product of $\grad\Hanz$ and $\grad\wass(\cf_\theh[\xvec],\cf_\theh[\xvec])$.
This is why our discretization of the $L^2$-Wasserstein metric by the norm $\nrm{\cdot}$
seems to be the right choice, if one is interested in a structure-preserving discretization.

\begin{cor}
The unique minimizier $\xvecm\in\xseqN$ of $\Halz$ is a minimizer of $\Falz$ and it holds for any $\xvec\in\xseqN$
\begin{align}\label{eq:FleqgradH}
	\Falz(\xvec) - \Falz(\xvecm) \leq \frac{2\alpha+1}{2}\nrm{\wgrad\Halz(\xvec)}^2.
\end{align}
\end{cor}
\begin{proof}
Equality \eqref{eq:dfeir} and $2\alpha-1\geq0$ shows, that $\xvec\mapsto\Falz(\xvec)$ is minimal, iff $\nrm{\wgrad\Halz(\xvec)}=0$ and $\Halz(\xvec)$ is minimal,
which is the case for $\xvec=\xvecm$. The representaion in \eqref{eq:dfeir} further implies
\begin{align*}
	&\Falz(\xvec)-\Falz(\xvecm) \\
	= &\nrm{\wgrad\Halz(\xvec)}^2 - \nrm{\wgrad\Halz(\xvecm)}^2 + (2\alpha-1)\La\big(\Halz(\xvec)-\Halz(\xvecm)\big) \\
	= &\nrm{\wgrad\Halz(\xvec)}^2 + (2\alpha-1)\La\big(\Halz(\xvec)-\Halz(\xvecm)\big) 
	\leq \left(1+\frac{2\alpha-1}{2}\right)\nrm{\wgrad\Halz(\xvec)}^2,
\end{align*}
where we used \eqref{eq:HleqF} in the last step.
\end{proof}

\subsubsection{Inerpretation of the scheme as discrete Wasserstein gradient flow}\label{sec:MM}
Starting from the discretized perturbed information functional $\Falz$
we approximate the spatially discrete gradient flow equation
\begin{align}
  \label{eq:sdgradflow}
  \dot{\xvec} = -\wgrad\Falz(\xvec)
\end{align}
also in time, using \emph{minimizing movements}.
To this end, remember the temporal decomoposition of $[0,+\infty)$ by 
\begin{align*}
	\left\lbrace 0 = t_0 < t_1 < \ldots < t_n < \ldots\right\rbrace,\quad \textnormal{where}\quad t_n = t_{n-1} + \tau_n,
\end{align*}
using time step sizes $\thet:=\{\tau_1,\tau_2,\ldots,\tau_n,\ldots\}$ with $\tau_n\leq\tau$ and $\tau>0$.
As before in the introduction, we combine the spatial and temporal mesh widths in a single discretization parameter $\Delta=(\thet;\theh)$.
For each $\yvec\in\xseqN$, introduce the \emph{Yosida-regularized information functional} 
$\Fy(\cdot,\cdot,\cdot,\yvec):[0,+\infty)\times[0,\tau]\times\xseqN$ by 
\begin{align}\label{eq:dmm}
  \Fy(\lambda,\sigma,\xvec,\yvec) = \frac{1}{2\sigma}\nrm{\xvec-\yvec}^2+\Falz(\xvec).
\end{align}
A fully discrete approximation $(\xvec_\Delta^n)_{n=0}^\infty$ of \eqref{eq:sdgradflow} is now defined inductively
from a given initial datum $\xvec_\Delta^0$ by choosing each $\xvec_\Delta^n$ 
as a global minimizer of $\Fy(\lambda,\tau_n,\cdot,\xvec_\Delta^{n-1})$.
Below, we prove that such a minimizer always exists (see Lemma \ref{lem:cfl}).

In practice, one wishes to define $\xvec_\Delta^n$ as --- preferably unique --- solution 
of the Euler-Lagrange equations associated to $\Fy(\lambda,\tau_n,\cdot,\xvec_\Delta^{n-1})$,
which leads to the implicit Euler time stepping:
\begin{align}
  \label{eq:euler}
  \frac{\xvec-\xvec_\Delta^{n-1}}{\tau} = -\wgrad\Falz(\xvec).
\end{align}
Using the explicit representation of $\grad\Falz$, 
it is immediately seen that \eqref{eq:euler} is indeed the same as \eqref{eq:dgf}.
Equivalence of \eqref{eq:euler} and the minimization problem is guaranteed at least for sufficiently small $\tau>0$,
as the following Proposition shows. 
%
\begin{prp}
  \label{prp:wellposed}
  For each discretization $\Delta$ and every initial condition $\xvec^0\in\xseqN$,
  the sequence of equations \eqref{eq:euler} can be solved inductively.
  Moreover, if $\tau>0$ is sufficiently small with respect to $\delta$ and $\Falz(\xvec^0)$,
  then each equation \eqref{eq:euler} possesses a unique solution with $\Falz(\xvec)\le\Falz(\xvec^0)$,
  and that solution is the unique global minimizer of $\Fy(\lambda,\tau_n,\cdot,\xvec_\Delta^{n-1})$.
\end{prp}
The proof of this proposition is a consequence of the following rather technical lemma.
\begin{lem}
  \label{lem:cfl}
  Fix a spatial discretization parameter $\delta$ and a bound $C>0$.
  Then for every $\yvec\in\xseqN$ with $\Falz(\yvec)\le C$, the following are true:
  \begin{itemize}
  \item for each $\sigma>0$, 
    the function $\Fy(\lambda,\sigma,\cdot,\yvec)$ possesses at least one global minimizer $\xvec^*\in\xseqN$;
  \item there exists a $\tau_C>0$ independent of $\yvec$ such that for each $\sigma\in(0,\tau_C)$,
    the global minimizer $\xvec^*\in\xseqN$ is strict and unique, 
    and it is the only critical point of $\Fy(\lambda,\sigma,\cdot,\yvec)$ with $\Falz(\xvec)\le C$.
  \end{itemize}
\end{lem}
\begin{proof}
Fix $\yvec\in\xseqN$ with $\Falz(\yvec)\leq C$, 
and define the nonempty (since it contains $\yvec$) sublevel $A_C:=\big(\Fy(\lambda,\sigma,\cdot,\yvec)\big)^{-1}([0,C])\subset\xseqN$.
First observe, that any $\xvec\in A_C$ holds
\begin{align}\label{eq:supp1}
	\sqrt{2\sigma C} \geq \nrm{\yvec-\xvec} \geq \delta^{\frac{1}{2}}\|\yvec-\xvec\|_{\infty} \geq \delta^{\frac{1}{2}}\big|\|\yvec\|_{\infty} - \|\xvec\|_{\infty}\big|,
\end{align}
hence $\|\xvec\|_{\infty}$ is bounded from above by $\sqrt{2\sigma C} + \|\yvec\|_\infty$. Especially,
\begin{align}\label{eq:supp2}
	\max_{k\in\ival} |x_k| \leq \sqrt{2\delta^{-1}\sigma C} + \|\yvec\|_{\infty} =: L(\delta,\sigma,\yvec),
\end{align}
which means, in the sense of density functions, that any $u=\cf_\theh[\xvec]$ with $\xvec\in A_C$ 
is compactly supported in $[-L(\delta,\sigma,\yvec),L(\delta,\sigma,\yvec)]$. 
Consequently, take again $\xvec\in\xseqN$ arbitrarily and declare 
$z_*=\min_{\kappa\in\hval} z_\kappa$ and $z^*=\max_{\kappa\in\hval} z_\kappa$, then on the one hand
the conservation of mass yields the boundedness of $z_*$ from above,
\begin{align*}
	1 = \intom \cf_\theh[\xvec]\dd x = \sum_{\kappa\in\hval} z_\kappa^{-1}(x_\kappp-x_\kappm) \leq 2 L(\delta,\sigma,\yvec) (z_*)^{-1},
\end{align*}
and on the other hand $\Falz(\xvec)\leq\Falz(\yvec)\leq C$ yields an upper bound for $z^*$, as the following calculation shows:
\begin{align}
	&(z^*)^{\alpha+\frac{1}{2}}-(z_*)^{\alpha+\frac{1}{2}} \leq \sum_{\kappa\in\hval}|z_\kappp^{\alpha+\frac{1}{2}} - z_\kappm^{\alpha+\frac{1}{2}}|
	\leq \left(\sum_{\kappa\in\hval}\delta\right)^{\frac{1}{2}}
	\left(\delta\sum_{\kappa\in\hval}\left(\frac{z_\kappp^{\alpha+\frac{1}{2}}- z_\kappm^{\alpha+\frac{1}{2}}}{\delta}\right)^2\right)^{\frac{1}{2}} \notag\\
	\Longrightarrow \;&
	z^*\leq \left(M\Theta_\alpha^{-1} C + \big(2L(\delta,\sigma,\yvec)\big)^{\alpha+\frac{1}{2}}\right)^{1/(\alpha+1/2)} \label{eq:zbound}.
\end{align}
Collecting the above observations, we first conclude that $A_C\subseteq\xseqN$ is a compact subset of $\setR^{K+1}$, 
due to $|x_0|,|x_K|\leq L(\delta,\sigma,\yvec)$ and the continuity of $\Falz$.
Moreover, every vector $\xvec\in A_C$ satisfies $x_\kappp-x_\kappm\ge\delta (z^*)^{-1}\geq \underline x$ for all $\kappa\in\hval$
with a positive constant $\underline x$ that depends on $C$ and $L(\delta,\sigma,\yvec)$.
Thus $A_C$ does not touch the boundary (in the ambient $\setR^{K+1}$) of $\xseqN$.
Consequently, $A_C$ is closed and bounded in $\xseqN$, endowed with the trace topology.

The restriction of the continuous function $\Fy(\lambda,\sigma,\cdot,\yvec)$ to the compact set $A_C$ 
possesses a minimizer $\xvec^*\in A_C$.
We clearly have $\Falz(\xvec^*)\le\Falz(\yvec)\le C$, 
and so $\xvec^*$ lies in the interior of $A_C$ and therefore is a global minimizer of $\Fy(\lambda,\sigma,\cdot,\yvec)$.
This proves the first claim.

Since $\Falz:\xseqN\to\setR$ is smooth, its restriction to $A_C$ is $\lambda_C$-convex with some $\lambda_C\le0$,
i.e., $\grad^2\Falz(\xvec)\ge\lambda_C\eins_{K+1}$ for all $\xvec\in A_C$.
Independently of $\yvec$, we have that
\begin{align*}
	\grad^2\Fy(\lambda,\sigma,\xvec,\yvec) = \grad^2\Falz(\xvec) + \frac\delta\tau\eins_{K+1},
\end{align*}
which means that $\xvec\mapsto\Fy(\lambda,\sigma,\xvec,\yvec)$ is strictly convex on $A_C$
if
\begin{align*}
	0< \sigma < \tau_C:=\frac\delta{(-\lambda_C)}.
\end{align*}
Consequently, each such $\Fy(\lambda,\sigma,\cdot,\yvec)$ has at most one critical point $\xvec^*$ in the interior of $A_C$, 
and this $\xvec^*$ is necessarily a strict global minimizer.
\end{proof}
\begin{rmk}[propagation of the support]
Take a solution $\xvec_\Delta$ of \eqref{eq:dmm} with density functions $u_\Delta$. 
As we already noted in the above proof, any density $u_\Delta^n$ 
has compact support in $[-L^n,L^n]$
with $L^n=L(\delta,\tau_n,\xvec_\Delta^{n-1})$ as in \eqref{eq:supp2}. Hence
\begin{align*}
	&\|\xvec_\Delta^n\|_{\infty} \leq \sqrt{2\delta^{-1}\tau_n\Falz(\xvec_\Delta^{n-1})} + \|\xvec_\Delta^{n-1}\|_{\infty} \\
	\Longrightarrow\;&
	\|\xvec_\Delta^n\|_{\infty} \leq \sqrt{2\delta^{-1}}\sum_{j=1}^n \sqrt{\tau_j\Falz(\xvec_\Delta^{j-1})} 
			+ \|x_\Delta^0\|_\infty,
\end{align*}
which is the best we can assume in case of $\lambda=0$. If $\lambda>0$, one can find a much better bound on the support of $u_\Delta$,
namely by replacing \eqref{eq:supp1} by,
\begin{align*}
	\|\xvec_\Delta^n\|_{\infty}\leq \delta^{-1}\nrm{\xvec_\Delta^n}^2 
	\leq \frac{2\delta^{-1}}{\lambda}\Falz(\xvec_\Delta^{n-1})
	\leq \frac{2\delta^{-1}}{\lambda}\Falz(\xvec_\Delta^0).
\end{align*}
\end{rmk}

%
\subsection{Some discrete variational theory}\label{sec:var}
In this section, we consider an arbitrary function $\Vz:\xseqN\to(-\infty,+\infty]$ and assume the existence of a value 
$\tau^*=\tau^*(\Vz)>0$, such that for any $\xvec\in\xseqN$ the minimization problem
\begin{align}\label{eq:mmV}
	\Vz_\sigma(\yvec,\xvec) := \frac{1}{2\sigma}\nrm{\yvec-\xvec}^2 + \Vz(\yvec) \quad\longrightarrow\quad\operatorname{min}
\end{align}
has a solution $\xvec_\sigma$ in $\xseqN$ for any $\sigma\in(0,\tau^*]$. 
In the literature, the function $\sigma\mapsto\xvec_\sigma$ is known as the \emph{De Giorgi's variational interpolant} 
connecting $\xvec$ and $\xvec_{\tau^*}$, see f.i. \cite[section 3.1]{AGS}.
Another interesting object in this context is the \emph{discrete Mureau-Yosida approximation}
$\sigma\mapsto\Vz_\sigma(\xvec_\sigma,\xvec)$, which can be defined for any $\xvec\in\xseqN$.
Analogously to the theory developed in \cite[section 3.1]{AGS}, one can even introduce a discrete version of the local slope of $\Vz$, i.e.
\begin{align}\label{eq:slp}
	\slp{\Vz}(\xvec) := \limsup_{\yvec\in\xseqN:\yvec\to\xvec}\frac{\big(\Vz(\xvec)-\Vz(\yvec)\big)^+}{\nrm{\xvec-\yvec}}
\end{align}
It turns out that $\slp{\Vz}(\xvec) = \nrm{\wgrad\Vz(\xvec)}$, which is a consequence of the lemma below
and an analogoue calculation done in \cite[Lemma 3.1.5]{AGS}. 

The above definitions remind of their continuous counterparts as defined in \cite[section 3.1]{AGS},
and the reader familiar with \cite{AGS} knows, that those objects are well-studied.
Some properties of the discrete Mureau-Yosida approximation, which will be needed in later sections
to study the asymptotic behaviour of solutions to \eqref{eq:dmm2}, are listened in the following lemma.
The proof is a special case of \cite[Theorem 3.1.4 and Lemma 3.1.5]{AGS}.
\begin{lem}\label{lem:Yosida_discrete}
Fix $\xvec\in\xseqN$ and declare by $\xvec_\sigma$ the De Giorgi's variational interpolant. 
Then it holds for any $\sigma\in(0,\tau^*]$ 
\begin{align}
\label{eq:Yosida_int}
	\frac{\nrm{\xvec_\sigma-\xvec}^2}{2\sigma} + \int_0^{\sigma}\frac{\nrm{\xvec_r-\xvec}^2}{2r^2}\dd r 
	= \Vz(\xvec) - \Vz(\xvec_\sigma).
\end{align}
If we further assume the continuity of
$\sigma\mapsto\grad\Vz(\xvec_\sigma)$ and $\sigma\mapsto\grad^2\Vz(\xvec_\sigma)$,
and the validity of the system Euler Lagrange equations
\begin{align*}
	\frac{1}{\sigma}(\xvec_\sigma-\xvec) = -\wgrad\Vz(\xvec_\sigma),
\end{align*}
for $\sigma\in(0,\tau^*]$, then
\begin{align}\label{eq:slprep}
	\nrm{\wgrad\Vz(\xvec)}^2 = \lim_{\sigma\downarrow0}\frac{\nrm{\xvec-\xvec_\sigma}^2}{\sigma^2}
	= \lim_{\sigma\downarrow0}\frac{\Vz(\xvec)-\Vz(\xvec_\sigma)}{\sigma}
	= \lim_{\sigma\downarrow0}\left(\frac{\Vz(\xvec)-\Vz(\xvec_\sigma)}{\nrm{\xvec-\xvec_\sigma}}\right)^2.
\end{align}
\end{lem}
\section{Analysis of equilibrium}\label{sec:equi}
%
%

In that which follows, we will analyze the long-time behaviour in the discrete setting and will especially prove Theorem \ref{thm:main2}.
As we have already seen in \cite{MMS}, the scheme's underlying variational structure is essential to get optimal decay rates. 
Due to our structure-preserving discretization, it is even possible to derive analogue, asymptotically equal decay rates for solutions to \eqref{eq:dmm}.

%
\subsection{Entropy dissipation -- the case of positive confinement $\lambda>0$}
In this section, we pursue the discrete rate of decay towards discrete equilibria and try to verify the statements in Theorem \ref{thm:main2}
to that effect. That is why we assume henceforth $\lambda>0$.

\begin{lem}\label{lem:exp}
A solution $\xvec_\Delta$ to the discrete minizing movement scheme \eqref{eq:dmm2} dissipates the entropies $\Halz$ and $\Falz$ at least exponential, i.e.
\begin{align}
	\big(1 + 2\tau_n\lambda\big)\left(\Halz(\xvec_\Delta^n) - \Halzm\right) &\leq \Halz(\xvec_\Delta^{n-1}) - \Halzm, 
	\quad\textnormal{and} \label{eq:expHn} \\
	\big(1+2\tau_n\lambda\big)\left(\Falz(\xvec_\Delta^n) - \Falzm\right)
	&\leq \Falz(\xvec_\Delta^{n-1}) - \Falzm \label{eq:expFn}
\end{align}
for any time step $n=1,2,\ldots$.
\end{lem}
\begin{proof}
Due to \eqref{eq:dfeir}, the gradient of the information functional $\Falz$ is given by
\begin{align*}
	\grad\Falz(\xvec) = 2\delta^{-1}(\grad\Halz(\xvec))^T\grad^2\Halz(\xvec) + (2\alpha-1)\La\grad\Halz(\xvec),
\end{align*}
which yields in combination with the $\La$-convexity of $\Halz$ and \eqref{eq:euler}
\begin{equation}\begin{split}\label{eq:step1}
	&\Halz(\xvec_\Delta^{n-1}) - \Halz(\xvec_\Delta^n) \\
	\geq& \tau_n\spr{\wgrad\Falz(\xvec_\Delta^n)}{\wgrad\Halz(\xvec_\Delta^n)} \\
	\geq& 2\tau_n\spr{\wgrad\Halz(\xvec_\Delta^n)}{\grad^2\Halz(\xvec)\wgrad\Halz(\xvec_\Delta^n)} + \tau_n(2\alpha-1)\La\nrm{\wgrad\Halz(\xvec_\Delta^n)}^2 \\
	\geq& 2\tau_n\La\nrm{\wgrad\Halz(\xvec_\Delta^n)}^2 + \tau_n(2\alpha-1)\La\nrm{\wgrad\Halz(\xvec_\Delta^n)}^2 
	\geq \tau_n(2\alpha+1)\La\nrm{\wgrad\Halz(\xvec_\Delta^n)}^2.
\end{split}\end{equation}
Using inequality \eqref{eq:HleqF}, we conclude in
\begin{align*}
	\big(1 + 2\tau_n(2\alpha+1)\La^2\big)\left(\Halz(\xvec_\Delta^n) - \Halz(\xvecm)\right) \leq \Halz(\xvec_\Delta^{n-1}) - \Halz(\xvecm)
\end{align*}
for any $n=1,2,\ldots$. Since $(2\alpha+1)\La^2=\lambda$, this shows \eqref{eq:expHn}.
To prove \eqref{eq:expFn}, we proceed as above. First introduce for $\sigma>0$ the vector $\xvec_\sigma$ as unique minimizier of
$\yvec\mapsto\frac{1}{2\sigma}\nrm{\yvec-\xvec_\Delta^n}^2+\Halz(\yvec)$. 
Then the fundamental property of the minimizing movement scheme reads as
\begin{align*}
	\Falz(\xvec_\Delta^n) - \Falz(\xvec_\sigma) 
	&\leq \frac{1}{2\tau_n}\left(\nrm{\xvec_\sigma-\xvec_\Delta^{n-1}}^2 - \nrm{\xvec_\Delta^n-\xvec_\Delta^{n-1}}^2 \right) \\
	&\leq \frac{1}{2\tau_n}\nrm{\xvec_\sigma-\xvec_\Delta^n}\left(\nrm{\xvec_\sigma-\xvec_\Delta^{n-1}} + \nrm{\xvec_\Delta^n-\xvec_\Delta^{n-1}} \right).
\end{align*}
Devide both side by $\sigma$ and pass to the limit as $\sigma\downarrow0$, we get
\begin{align*}
	\spr{\wgrad\Falz(\xvec_\Delta^n)}{\wgrad\Halz(\xvec_\Delta^n)} 
	\leq \frac{1}{\tau_n}\nrm{\wgrad\Halz(\xvec_\Delta^n)}\nrm{\xvec_\Delta^n-\xvec_\Delta^{n-1}},
\end{align*}
due to \eqref{eq:slprep}, and thanks to \eqref{eq:step1} further
\begin{align*}
	\tau_n(2\alpha+1)\La\nrm{\wgrad\Halz(\xvec_\Delta^n)}^2 \leq \nrm{\wgrad\Halz(\xvec_\Delta^n)}\nrm{\xvec_\Delta^n-\xvec_\Delta^{n-1}}.
\end{align*}
As a consequence, we get two types of inequalities, namely
\begin{equation}\begin{split}\label{eq:step2}
	\tau_n\sqrt{(2\alpha+1)\lambda}\nrm{\wgrad\Halz(\xvec_\Delta^n)} &\leq \nrm{\xvec_\Delta^n-\xvec_\Delta^{n-1}}
	\quad\textnormal{and} \\
	2\tau_n^2\lambda\big(\Falz(\xvec_\Delta^n) - \Falz(\xvecm)\big)
	&\leq \nrm{\xvec_\Delta^n-\xvec_\Delta^{n-1}}^2,
\end{split}\end{equation}
where we used $\La=\sqrt{\lambda/(2\alpha+1)}$ and \eqref{eq:FleqgradH}.
To get the desired estimate, remember the De Giorgi's variational interpolation:
Fix $\xvec_\Delta^{n-1}$ and denote by $\xvec_\sigma^n$ a minimizer of 
$\yvec\mapsto\frac{1}{2\sigma}\nrm{\yvec-\xvec_\Delta^{n-1}}^2 + \Falz(\yvec)$ for $\sigma\in(0,\tau_n]$. 
Then $\xvec_\sigma^n$ connects $\xvec_\Delta^{n-1}$ and $\xvec_\Delta^n$ and the monotonicity of $\sigma\mapsto\Falz(\xvec_\sigma^n)$
and \eqref{eq:step2} yields for any $\sigma\in(0,\tau]$
\begin{equation}\begin{split}\label{eq:step3}
	2\sigma^2\lambda\big(\Falz(\xvec_\Delta^n) - \Falz(\xvecm)\big)
	&\leq 2\tau_n^2\lambda\big(\Falz(\xvec_\sigma^n) - \Falz(\xvecm)\big) \\
	&\leq \nrm{\xvec_\sigma^n-\xvec_\Delta^{n-1}}^2.
\end{split}\end{equation}
Further remember the validity of \eqref{eq:Yosida_int} in Lemma \ref{lem:Yosida_discrete}, which gives us in this special case
\begin{align*}
	\Falz(\xvec_\Delta^n) + \frac{\nrm{\xvec_\Delta^n-\xvec_\Delta^{n-1}}^2}{2\tau_n} 
	+ \int_0^{\tau_n}\frac{\nrm{\xvec_\sigma^n-\xvec_\Delta^{n-1}}^2}{2\sigma^2}\dd \sigma 
	= \Falz(\xvec_\Delta^{n-1}).
\end{align*}
Inserting \eqref{eq:step3} in the above equation then finally yields
\begin{align*}
	\left(1+2\tau_n\lambda\right)\big(\Falz(\xvec_\Delta^n) - \Falz(\xvecm)\big)
	\leq \Falz(\xvec_\Delta^{n-1}) - \Falz(\xvecm),
\end{align*}
and the claim is proven.
\end{proof}
\begin{rmk}
In the continuous situation, the analogue proofs of \eqref{eq:expHn} and \eqref{eq:expFn} 
require a more deeper understanding of variational techniques. An essential tool in this context is the 
\emph{flow interchange lemma}, see f.i. \cite[Theorem 3.2]{MMS}.
Although one can easily proof a discrete counterpart of the flow interchange lemma, it is not essential in the above proof, 
since the smoothness of $\xvec\mapsto\Halz(\xvec)$ allow an explict calculation of its gradient and hessian.
\end{rmk}
Lemma \ref{lem:exp} paves the way for the exponential decay rates of Theorem \ref{thm:main2}. 
Evectively, \eqref{eq:expH} and \eqref{eq:expF} are just applications of the following version of 
the dicrete Gronwall lemma: Assume $\{c_n\}_{n\in\setN}$ and $\{y_n\}_{n\in\setN}$ to be sequences with values in $\setR_+$, satisfying 
$(1+c_n)y_n \leq y_{n-1}$ for any $n\in\setN$, then 
\begin{align*}
	y_n \leq y_0 e^{-\sum_{k=0}^{n-1} \frac{c_k}{1+c_k}},\quad\textnormal{for any } n\in\setN.
\end{align*}
This statement can be easily proven by induction. 
Furthermore, inequality \eqref{eq:CK} is then a corollary of \eqref{eq:expH} and a Csiszar-Kullback inequality, see \cite[Theorem 30]{CaJuMa}.

%
\subsubsection{Convergence towards Barenblatt profiles and Gaussians}
Assume again $\lambda>0$. 
It is an striking fact (see \cite{DMfourth}), that the stationary solutions $\uinf$ and $\vinf$ of \eqref{eq:fofo} and \eqref{eq:heat}, respecively, are identical. 
Those stationary states 
have the form of Barenblatt profils or Gaussians, respectively,
\begin{gather*}
	\bal = \big(a-b|x|^2\big)_+^{1/(\alpha-1/2)},\quad b = \frac{\alpha-1/2}{\sqrt{2\alpha}}\La \quad\textnormal{if } \alpha> 1/2 \textnormal{ and}\\
	\bhl = a e^{-\Lh|x|^2} \quad\textnormal{if } \alpha=1/2,
\end{gather*}
where $a\in\setR$ is chosen to conserve unit mass. 

To prove the statement of Theorem \ref{thm:main1}, we are going to show that 
the sequence of functionals $\Halzd{\cdot}:\dens\to(-\infty,+\infty]$ given by
\begin{align*}
	\Halzd{u} := \begin{cases} \Hal(u) &, u\in\densN \\ +\infty &,\textnormal{else}\end{cases}
\end{align*}
$\Gamma$-congerves towards $\Hal$. More detailed, for any $u\in\dens$ it holds
\begin{enumerate}
	\item[(i)] $\liminf_{\delta\to0} \Halzd{u_\delta}\geq\Hal(u)$ for any sequence $u_\delta$ with $\lim_{\delta\to0}\wass(u_\delta,u)=0$.
	\item[(ii)] There exists a \emph{recovery sequence} $u_\delta$ of $u$, i.e. $\limsup_{\delta\to0} \Halzd{u_\delta}\leq\Hal(u)$
				and $\lim_{\delta\to0}\wass(u_\delta,u)=0$.
\end{enumerate}
The $\Gamma$-convergence of $\Halzd{\cdot}$ towards $\Hal$ is a powerful property, 
since it implies convergence of the sequence of minimizers $\um = \cf_\theh[\xvecm]$ towards $\bal$ or $\bhl$, repsectively, w.r.t. the $L^2$-Wasserstein metric,
see \cite{Braides}.
To conclude even strong convergence of $\um$ at least in $L^p(\Omega)$ for arbitrary $p\geq 1$,
we proceed similar as in \cite[Proposition 18]{dde}.
Necessary for that, recall that the total variation of a function $f\in L^1(\Omega)$ is given by
\begin{align}
  \label{eq:defTV}
  \tv{f} := 
	\sup \bigg\{ \intom f(x)\varphi'(x)\dd x \,\bigg|\, \varphi\in\operatorname{Lip}(\Omega)\textnormal{ with compact support},\,\sup_{x\in\Omega} |\varphi(x)|\le 1\bigg\},
\end{align}
we refer \cite[Definition 1.1]{Giusti}.
If $f$ is a piecewise constant function with compact support $[x_0,x_K]$, taking values $f_\kmh$ on intervals $(x_{k-1},x_k]$,
then the integral in \eqref{eq:defTV} amounts to
\begin{align*}
  \intom f(x)\varphi'(x)\dd x = \sum_{k=1}^K \big[f(x)\varphi(x)\big]_{x=x_{k-1}+0}^{x_k-0}
  = \sum_{k=1}^{K-1} (f_\kmh-f_\kph)\varphi(x_k) + f_\imh\varphi(x_0) - f_\Kmh\varphi(x_K).
\end{align*}
Consequently, for such $f$, the supremum in \eqref{eq:defTV} equals
\begin{align}
  \label{eq:ttv}
  \tv{f} = \sum_{k=1}^{K-1}|f_\kph-f_\kmh| + |f_\imh| + |f_\Kmh|
\end{align}
and is attained 
for every $\varphi\in\operatorname{Lip}(\Omega)$ with $\varphi(x_k)=\operatorname{sgn}(f_k-f_{k+1})$ at $k=1,\ldots,K-1$,
$\varphi(x_0) = \operatorname{sgn}(f_\imh)$ and $\varphi(x_K) = -\operatorname{sgn}(f_\Kmh)$.
\begin{lem}\label{lem:Lconv}
For any $\alpha\in[\frac{1}{2},1]$, 
assume $\xvecm\in\xseqN$ to be the unique minimizer of $\Halz$ and declare the sequence of functions $\um = \cf_\theh[\xvecm]$.
Then
\begin{align}
	\um \xrightarrow{\delta\to0} \bal &,\textnormal{ strongly in } L^p(\Omega) 
	\textnormal{ for any } p\geq 1, \label{eq:Lconv2} \\
	\umh \xrightarrow{\delta\to0} \bal &,\textnormal{ uniformly on } \Omega, \label{eq:uniform2}
\end{align}
where $\umh:\Omega\to\setR$ is a local affine interpolation of $\um$ on $\Omega$, such that for any $\kappa\in\hval\cup\ival$
\begin{align*}
	\big(\umh\circ\cX_\theh[\xvecm]\big)(\xi_\kappa) 
	=  z_\kappa := \begin{cases} z_\kappa, &\kappa\in\hval \\ 
															\frac{1}{2}(z_\kappp+z_\kappm),& \kappa\in\ival \end{cases}.
\end{align*}
\end{lem}
\begin{proof}
We will first prove the $\Gamma$-convergence of $\Halzd{\cdot}$ towards $\Hal$.
The first requirement $(i)$ is a trivial conclusion of the the lower semi-continuity of $\Hal$. 
For the second point $(ii)$, we fix $u\in\dens$ and assume $\theX:[0,M]\to[-\infty,+\infty]$ 
to be the Lagrangian map of $u$. Further introduce the projection map $\pi_\theh:\xspc\to\xspc_\theh$ on the space of Lagrangian maps, 
\begin{align*}
	\pi_\theh[\theX] = \sum_{k=0}^K \theX(\xi_k)\hatf_k(\xi), 
\end{align*}
where $\hatf_\kappa: [0,M]\to\Omega$ are local affine hat-functions with $\hatf_\kappa(\iota\delta)=\delta_\kappa^\iota$ for any $\kappa,\iota\in\hval\cup\ival$.
We claim that the corresponding sequence $u_\delta$ to $\pi_\theh[\theX]$ is the right choice for the recovery sequence. 
To prove the convergence in the $L^2$-Wasserstein metric, we fix $\eps>0$
and take a compact set $\mathcal{K}\subseteq[-L,L]\subseteq\Omega$ with $\int_{\mathcal{K}}|x|^2u(x)\dd x <\eps$ and $\int_{\mathcal{K}}|x|^2u_\theh(x)\dd x <\eps$,  
which can be done due to the boundedness of the second momentum.
Since $\theX$ and $\pi_\theh[\theX]$ are monotonically increasing, it holds for any $\xi\in [0,M]$
\begin{align*}
	|\theX(\xi)-\pi_\theh[\theX](\xi)| \leq \big(\theX(\xi_k)-\theX(\xi_{k-1})\big), 
	\textnormal{ with } \xi\in[\xi_{k-1},\xi_k], k=1,\ldots,K
\end{align*}
and further for $\delta\leq \eps(2L)^{-2}$
\begin{align*}
	&\|\theX-\pi_\theh[\theX]\|_{L^2([0,M])}^2 \\
	\leq &\|\theX-\pi_\theh[\theX]\|_{L^2(\theX^{-1}([-L,L]))}^2 + \|\theX-\pi_\theh[\theX]\|_{L^2([0,M]\backslash\theX^{-1}([-L,L]))}^2\\
	\leq &2L\delta\sum_{\substack{k=1,\ldots,K \\ \theX(\xi_k)\in[-L,L]}} \big|\theX(\xi_k)-\theX(\xi_{k-1})\big| \\
			 &+ \left(\|\theX\|_{L^2([0,M]\backslash\theX^{-1}([-L,L]))}
			 + \|\pi_\theh[\theX]\|_{L^2([0,M]\backslash\theX^{-1}([-L,L]))}\right)^2 \\
	\leq &4L^2\delta + 8\eps \leq 9\eps.
\end{align*}
This shows $u_\delta\to u$ in $L^2$-Wasserstein, as $\delta\to0$. 
The second point $(2)$ easily follows by using Jensen's inequality,
\begin{align*}
	\Halzd{u_\delta} &= \Hal(u_\delta) = \sum_{\kappa\in\hval}\int_{x_{\kappm}}^{x_\kappp} \phia\left(\frac{\delta}{x_\kappp-x_\kappm}\right) \dd x \\
	&= \sum_{\kappa\in\hval} (x_\kappp-x_\kappm) \phia\left(\frac{1}{x_\kappp-x_\kappm}\int_{x_\kappm}^{x_\kappp} u(s)\dd s\right) \\
	&\leq \sum_{\kappa\in\hval} \int_{x_\kappm}^{x_\kappp} \phia\big(u(s)\big)\dd s
	= \Hal(u).
\end{align*}
Taking the limes superior on both sides proves $\limsup_{\delta\to0}\Halzd{u_\delta} \leq \Hal(u)$ and since $\Hal$ is lower semi-continuous, 
we especially obtain $\lim_{\delta\to0}\Halzd{u_\delta} = \Halz(u)$.
Due to the the equi-coercivity of $\Halzd{\cdot}$
and $\inf_{u\in\dens}\Halzd{u}=\Halz(\xvecm)$, $\um$ converges towards $\bal$ in w.r.t. $\wass$, by \cite[Theorem 1.21]{Braides}

The convergence of $\Halz(\xvecm)$ to $\Hal(\bal)$ yields on the one hand the uniform boundedness of $\Halz(\xvecm)$ w.r.t. 
the spatial discretization parameter $\delta$, and on the other hand the 
uniform boundedness of $\Falz(\xvecm)$, which is a conclusion of \eqref{eq:dfeir} and $\wgrad\Halz(\xvec)=0$.
Similar to \cite[Proposition 18]{dde}, one can easily prove that the term $\Falz(\xvecm)$ is an upper bound on 
the total variation of $\prss(\um)$ with $\prss(s):=\Theta_\alpha s^{\alpha+1/2}$:
Take any arbitrary $\yvec\in\setR^{K+1}$ with $\|\yvec\|_{\infty}\leq1$ and define $\cY = \cX_\theh[\yvec]$. Then
\begin{align}\label{eq:TVstep1}
	\spr{\wgrad\Hanz(\xvecm)}{\yvec} = \spr{\wgrad\Halz(\xvecm)}{\yvec} - \La\spr{\xvecm}{\yvec},
\end{align}
and the left hand side can be reformulated, using \eqref{eq:gradH} and a change of variables, i.e.
\begin{align*}
	\spr{\wgrad\Hanz(\xvecm)}{\yvec}
	&= -\delta\sum_{\kappa\in\hval}\prss(z_\kappa)\frac{y_\kappp-y_\kappm}{\delta}
	= -\int_0^M \prss(\um\circ\cX_\Delta^{\operatorname{min}}) \partial_\xi \cY \dd\xi \\
	&= -\int_{x_0}^{x_K} \prss(\um) \widetilde{\varphi}_x \dd x
\end{align*}
with the Lipschitz-continuous function $\widetilde{\varphi}:[x_0,x_K]\to[-1,1]$ defined by 
$\widetilde{\varphi}(x) := \cX_\theh[\yvec]\circ(\cX_\Delta^{\operatorname{min}})^{-1}$. 
Since $\um$ is equal to zero in $\Omega\backslash[x_0,x_K]$, we can define any extension $\varphi:\Omega\to[-1,1]$ of $\widetilde{\varphi}$
with compact support, and exchange it in the above calculation without changing the value of the integral. 
Moreover, by the Cauchy-Schwarz inequality, $\nrm{\yvec}\leq\|\yvec\|_\infty\leq 1$ and \eqref{eq:TVstep1}, 
\begin{align*}
	\intom \prss(\um) \varphi_x \dd x \leq \nrm{\wgrad\Halz(\xvecm)} + \La\nrm{\xvecm},
\end{align*}
which is uniformly bounded from above, due to \eqref{eq:dfeir} and the uniform boundedness of $\Halz(\xvecm)$ and $\Falz(\xvecm)$. 
This proves the uniform boundedness of $\tv{\prss(\um)}$ and 
using the superlinear growth of $s\mapsto\prss(s)$ and \cite[Proposition 1.19]{Giusti}, we conclude in \eqref{eq:Lconv2}.
 
To proof \eqref{eq:uniform2}, we show that the $H^1(\Omega)$-norm of $\umh$ is bouned by the information functional.
This was already done in \cite{dlssv3} for $\alpha=\frac{1}{2}$, where we also showed
\begin{align*}
	\|\umh\|_{H^1(\Omega)}^2 = \delta\sum_{k\in\ival} \frac{z_\kph+z_\kmh}{2}\left(\frac{z_\kph - z_\kmh}{\delta}\right)^2.
\end{align*}
So assume $\alpha\in(\frac{1}{2},1]$, then the concavity of the mapping $s\mapsto s^{\alpha-1/2}$ yields
for any values $b>a>0$
\begin{align*}
	b^{\alpha+\frac{1}{2}} - a^{\alpha+\frac{1}{2}} = (\alpha+\tfrac{1}{2})\int_a^b s^{\alpha-\frac{1}{2}} \dd s
	\geq (\alpha+\tfrac{1}{2}) \frac{b^{\alpha-\frac{1}{2}} + a^{\alpha-\frac{1}{2}}}{2}(b-a),
\end{align*}
and further $(b^{\alpha+\frac{1}{2}} - a^{\alpha+\frac{1}{2}})^2 \geq (\alpha+\tfrac{1}{2})^2 \frac{b^{2\alpha-1} + a^{2\alpha-1}}{4}(b-a)^2$.
Therefore
\begin{align*}
	\|\umh\|_{H^1(\Omega)}^2 
	&= \delta\sum_{k\in\ival} \frac{z_\kph+z_\kmh}{2}\left(\frac{z_\kph - z_\kmh}{\delta}\right)^2 \\
	&\leq \|\umh\|_{L^\infty(\Omega)}^{2(1-\alpha)}\delta
			\sum_{k\in\ival}\frac{z_\kph^{2\alpha-1}+z_\kmh^{2\alpha-1}}{2}\left(\frac{z_\kph - z_\kmh}{\delta}\right)^2 \\
	&\leq \frac{2 \|\umh\|_{L^\infty(\Omega)}^{2(1-\alpha)}}{(\alpha+\tfrac{1}{2})^2}\delta
			\sum_{k\in\ival}\left(\frac{z_\kph^{\alpha+\frac{1}{2}} - z_\kmh^{\alpha+\frac{1}{2}}}{\delta}\right)^2 
	\leq \frac{4\alpha \|\umh\|_{L^\infty(\Omega)}^{2(1-\alpha)}}{(\alpha+\tfrac{1}{2})^2}\Falz(\xvecm).
\end{align*}
\end{proof}

%
\subsection{The case of zero confinement $\lambda=0$}\label{sec:confn}
We will now consider equation \eqref{eq:fofo} in case of vanishing confinement $\lambda=0$, hence
\begin{align}\label{eq:fofo0}
	\partial_t u = -\big(u (u^{\alpha-1}u_{xx}^\alpha)_x\big)_x, \quad\textnormal{for } (t,x)\in(0,+\infty)\times\Omega,
\end{align}
and $u(0)=u^0$ for arbitrary initial density $u^0\in\dens$.
From the continuous theory, it is known that solutions to \eqref{eq:fofo0} or \eqref{eq:heat} with $\La=0$ 
branches out over the whole set of real numbers $\Omega=\setR$, hence converges towards zero 
at a.e. point. 
This matter of fact makes rigorous analysis of the long-time behaviour of solutions to \eqref{eq:fofo0} 
more difficult as in the case of positive confinement.
However, the unperturbed functionals $\Han$ and $\Fan$ hold the scaling property, see again \cite{MMS},
\begin{align}\label{eq:scaling}
	\Han(\ds_r u) = r^{-(2\alpha-1)/2}\Han(u),\quad\textnormal{and}\quad 
	\Fan(\ds_r u) = r^{-(2\alpha+1)}\Fan(u),
\end{align}
for any $r>0$ and $\ds_r u(x) := r^{-1} u(r^{-1}\cdot)$ with $u\in\dens$. 
Due to this, it is possible to find weak solutions to a rescaled version of \eqref{eq:fofo0} 
by solving problem \eqref{eq:fofo} with $\lambda=1$. More detailed, it holds the following Lemma:
\begin{lem}
A function $u\in\Lloc^2((0,T);W^{2,2}(\Omega))$ is a weak solution of \eqref{eq:fofo} with $\lambda=1$, iff
\begin{align}\label{eq:defR}
	w(t,\cdot) = \ds_{R(t)} u(\log(R(t),\cdot),\quad\textnormal{with}\quad 
	R(t) := \big(1+(2\alpha+3) t\big)^{1/(2\alpha+3)}
\end{align}
is a weak solution to \eqref{eq:fofo0}.
\end{lem}
A consequence of the above Lemma is, that one can describe how solutions $w$ to \eqref{eq:fofo0} 
vanishes asymptotically as $t\to\infty$, although the gained information is not very strong and useful:
In fact the first observation (without studying local asymptotics in more detail) is, 
that $w$ decays to zero with the same rate as the rescaled (time-dependent) Barenblatt-profile $\bant$ defined by $\bant(t,\cdot) := \ds_{R(t)}\bai$,
with $R(t)$ of \eqref{eq:defR}. It therefore exists a constant $C>0$ just depending on $\Han(w^0)=\Han(u^0)$ with
\begin{align}\label{eq:Rdecay}
	\|w(t,\cdot) - \bant(t,\cdot)\|_{L^1(\Omega)} \leq C R(t)^{-1},
\end{align}
for any $t>0$. 
In \cite{MMS}, this behaviour was described using weak solutions constructed by minimizing movements. 
We will adopt this methodes to derive a discrete analogue of \eqref{eq:Rdecay} for our discrete solutions $\xvec_\Delta$ of \eqref{eq:dmm}.

First of all, we reformulate the scaling operator $\ds_r$ for fixed $r>0$ in the setting of monotonically increasing vectors $\xvec\in\xseqN$.
Since $\ds_r u(x) := r^{-1} u(r^{-1}\cdot)$ for arbitrary density in $\dens$, the same can be done for $u_\delta=\cf_\theh[\xvec]$, hence
\begin{align*}
	\ds_r u_\delta(x) = \sum_{k=1}^K \frac{r^{-1}\delta}{x_k-x_{k-1}}\indy_{(x_\kappm,x_\kappp]}(r^{-1}x)
	= \sum_{k=1}^K \frac{\delta}{rx_k-rx_{k-1}}\indy_{(rx_\kappm,rx_\kappp]}(x)
	= \cf_\theh[r\xvec](x)
\end{align*}
for any $x\in\Omega$. The natural extension of $\ds_r$ to the set $\xseqN$ is hence 
\begin{align*}
	\ds_r\xvec := r\xvec ,\quad\textnormal{with corrseponding}\quad \ds_r\zvec = \cz_\theh[\ds_r\xvec] = r^{-1}\zvec.
\end{align*}
As a consequence of this definition, we note that it holds the discrete scaling property for $\Hanz$ and $\Fanz$, i.e.
for any $r>0$ and $\xvec\in\xseqN$,
\begin{align}\label{eq:dscaling}
	\Hanz(\ds_r\xvec) = r^{-(2\alpha-1)/2}\Hanz(\xvec),\quad\textnormal{and}\quad 
	\Fanz(\ds_r\xvec) = r^{-(2\alpha+1)}\Fan(\xvec).
\end{align}
The first equality holds, due to $\Hanz(\xvec) = \Han(\cf_\theh[\xvec])$ and the scaling property \eqref{eq:scaling} of the continuous entropy functions.
The analogue claim for $\Fanz$ in \eqref{eq:dscaling} follows by inserting $\ds_r\xvec$ into $\grad\Hanz$
and using $\ds_r\zvec = r^{-1}\zvec$, then
\begin{align*}
	&\grad\Hanz(\ds_r\xvec) = \Theta_\alpha \delta\sum_{\kappa\in\hval}\big(\ds_r\zvec_\kappa\big)^{\alpha+\frac{1}{2}}\frac{\ee_\kappm-\ee_\kappp}{\delta}
	= r^{-(\alpha+1/2)}\grad\Hanz(\xvec) \\
	\Longrightarrow\;&
	\Fanz(\ds_r\xvec) = \nrm{\wgrad\Hanz(\ds_r\xvec)}^2 = r^{-(2\alpha+1)}\nrm{\wgrad\Hanz(\xvec)}^2 = r^{-(2\alpha+1)}\Fanz(\xvec).
\end{align*}

This scaling properties can now be used to build a bridge between solutions of discrete minimizing movement schemes with $\lambda=0$
to those with positive confinement. The following Lemma is based on the proof of Theorem \cite[Theorem 5.6]{MMS},
but nevertheless, it is an impressive example for the powerful structure-preservation of our discrete scheme.
\begin{lem}\label{lem:dmm_scaling}
Assume $\xvec^*\in\xseqN$ with $\Falz(\xvec^*)<\infty$. 
Further fix $\tau>0$ and $R>S>0$. 
Then $\xvec\in\xseqN$ is a minimizer of
\begin{align}\label{eq:dmm3}
	\yvec\mapsto\Fy(\lambda,\tau,\yvec,\xvec^*) = \frac{1}{2\tau}\nrm{\yvec-\xvec^*}^2 + \Fanz(\yvec) + \frac{\lambda}{2}\nrm{\yvec}^2,
\end{align}
if and only if $\ds_R\xvec\in\xseqN$ minimizes the functional
\begin{gather}\begin{split}\label{eq:dmmds}
	\wvec\mapsto\Fy(\lt,\taut,\wvec,\ds_S\xvec^*) = \frac{1}{2\taut}\nrm{\wvec-\ds_S\xvec^*}^2 + \Fanz(\wvec) + \frac{\lt}{2}\nrm{\wvec}^2, 
	\quad\textnormal{with}, \\
	\taut = \tau S R^{2\alpha+2},\quad \lt = \frac{S(1+\lambda\tau) - R}{\taut R}
\end{split}\end{gather}
\end{lem}
It is not difficult so see, that this Lemma holds for all functionals $\Vz:\xseqN\to\setR$ with the same scaling property as $\Fanz$ in \eqref{eq:dscaling}.

\begin{proof2}{Lemma \ref{lem:dmm_scaling}}
To simplify the proof, we show first that we can assume $S=1$ without loss of generality,
which is because of the following calculation: 
If for $R>S>0$ the vector $\ds_R\xvec$ minimizes \eqref{eq:dmmds}, 
then the linearity of $\nrm{\cdot}$ and \eqref{eq:dscaling} yields
\begin{align*}
	\Fy(\lt,\taut,\ds_R\xvec,\ds_S\xvec^*)
	&= \frac{S^2}{2\taut}\nrm{S^{-1}\ds_R\xvec-\xvec^*}^2 + S^{-2\alpha+1}\Fanz(S^{-1}\ds_R\xvec) + S^2\frac{\lt}{2}\nrm{S^{-1}\ds_R\xvec}^2 \\
	&= S^{-(2\alpha+1)}\left(\frac{1}{2\taut S^{-(2\alpha+3)}}\nrm{\ds_{\widetilde{R}}\xvec-\xvec^*}^2 
	+ \Fanz(\ds_{\widetilde{R}}\xvec) + S^{2\alpha+3}\frac{\lt}{2}\nrm{\ds_{\widetilde{R}}\xvec}^2\right) \\
	&= S^{-(2\alpha+1)}\Fy(\widetilde{\lambda},\widetilde{\tau},\ds_{\widetilde{R}}\xvec,\xvec^*),
\end{align*}
with $\widetilde{R} = \frac{R}{S} > 1 > 0$ and the new constants
\begin{align*}
	\widetilde{\tau} = \tau S R^{2\alpha+2} S^{-(2\alpha+3)} = \tau \widetilde{R}^{2\alpha+3}, \quad\textnormal{\and}\quad
	\widetilde{\lambda} = S^{2\alpha+3}\frac{(1+\lambda\tau) - R/S}{\taut R/S} = \frac{(1+\lambda\tau) - \widetilde{R}}{\widetilde{\tau} \widetilde{R}},
\end{align*}
hence $\ds_{\widetilde{R}}\xvec$ minimizes $\Fy(\widetilde{\lambda},\widetilde{\tau},\ds_{\widetilde{R}}\xvec,\xvec^*)$.

So assume $S=1$ and $R>1$ in \eqref{eq:dmmds} by now. Further introduce the functional $g:\xseqN\times\setR\to\setR$
\begin{align*}
	g(\yvec,r) := \frac{1}{2}\nrm{\ds_r\yvec - \xvec^*}^2 + r\Fan(\yvec) + \frac{r}{2}(1+\lambda\tau-r)\nrm{\yvec}^2 ,
\end{align*}
then by definition
\begin{align}\label{eq:gr}
	\tau^{-1}g(\yvec,1) = \Fy(\lambda,\tau,\yvec,\xvec^*),\quad\textnormal{and}\quad
	(\tau R^{2\alpha+2})^{-1}g(\yvec,R) = \Fy(\lt,\taut,\ds_R\yvec,\xvec^*).
\end{align}
For fixed $\yvec\in\xseqN$, a straight-forward calculation shows, that the derivative of $r\mapsto g(\yvec,r)$ holds
\begin{align*}
	\partial_r g(\yvec,r) 
	&= \spr{\ds_r\yvec - \xvec^*}{\yvec} + \Fan(\yvec) - \frac{r}{2}\nrm{\yvec}^2 + \frac{1}{2}(1+\lambda\tau-r)\nrm{\yvec}^2 \\
	&= -\spr{\xvec^*}{\yvec} + \Fan(\yvec) + \frac{1}{2}(1+\lambda\tau)\nrm{\yvec}^2
	= \frac{1}{2}\nrm{\yvec - \xvec^*}^2 - \frac{1}{2}\nrm{\xvec^*}^2 + \Fan(\yvec) + \frac{\lambda\tau}{2}\nrm{\yvec}^2 \\
	&= g(\yvec,1) - \frac{1}{2}\nrm{\xvec^*}^2.
\end{align*}
Hence, if $\xvec$ minimizes \eqref{eq:dmm3}, then the same vector minimizes $\yvec\mapsto g(\yvec,1)$ and further $\yvec\mapsto\partial_r g(\yvec,r)$
for any $r>0$. By integration
\begin{align*}
	&g(\yvec,r) - g(\yvec,1) = \int_1^r \partial_s g(\yvec,s) \dd s = (r-1)(g(\yvec,1) - \frac{1}{2}\nrm{\xvec^*}^2) \\
	\Longrightarrow\;&
	g(\yvec,r) = rg(\yvec,1) - (r-1)\frac{1}{2}\nrm{\xvec^*}^2
\end{align*}
for any $r>1$ and $\yvec\in\xseqN$. This means especially, that for arbitrary $r>1$, the function $g(\yvec,r)$ is minimal if and only if $g(\yvec,1)$ 
is so. This proves in combination with \eqref{eq:gr}, that $\ds_R\xvec$ is a minimizer of \eqref{eq:dmmds}.
By integration of $\partial_s g(\yvec,s)$ over $[r^{-1},1]$, $r>1$, one can analogously prove, that if $\widehat{\xvec}\in\xseqN$ is a minimizer of \eqref{eq:dmmds},
the rescaled vector $\ds_{R^{-1}}\widehat{\xvec}$ has to be a minimizer of \eqref{eq:dmm3}.
\end{proof2}

Before we prove the claim of Theorem \ref{thm:main3}, let us introduce the rescaled discrete Barenblatt-profile.
Define inductively for $n=0,1,\ldots$
\begin{align}\label{eq:Sn}
	\St^0 := 1,\quad \St^n = (1+\tau_n)\St^{n-1}.
\end{align}
Further take the minimizer $\xvecm\in\xseqN$ of the functional $\xvec\mapsto\Haiz(\xvec)$. 
Then denote the scaled vector $\bvec_{\Delta,\alpha,0}^n := \ds_{\St^n}\xvecm$ and define its corresponding density function
$\bantn = \cf_\theh[\bvec_{\Delta,\alpha,0}^n]$. This function can be interpreted as a self-similar solution of \eqref{eq:dmmds} with
initial density $\cf_\theh[\xvecm]$, $\lt=0$ and with time steps $\taut_n$ inductively defined by $\taut_n:=\tau_n\St^{n-1}\St^{2\alpha+2}$.

\begin{proof2}{Theorem \ref{thm:main3}}
As already mentioned above, we define a sequence of functions $\St^n$ inductively trough \eqref{eq:Sn} and declare 
a new partition of the time scale $[0,+\infty)$ by
\begin{align}\label{eq:partitiontaut}
	\{0=\sh_0<\sh_2<\ldots<\sh_n<\ldots\},\quad\textnormal{where } \sh_n:=\sum_{k=1}^n \taut_k
	\textnormal{ and } \taut_k:=\tau_k\St^{k-1}(\St^k)^{2\alpha+2},
\end{align}
and we write $\thett=(\taut_1,\taut_2,\ldots)$.
As a first conequence of the iterative character of the above object, we note that $(1+x)\leq e^x$ causes $\St^n \leq e^{t_n}$
for any $n=0,1,\ldots$. Moreover
\begin{align*}
	\sh_n = \sum_{k=1}^n \tau_k\St^{k-1}(\St^k)^{2\alpha+2}
	= \sum_{k=1}^n \tau_k(1+\tau_k)^{2\alpha+2}(\St^{k-1})^{2\alpha+3}
	\leq (1+\tau)^{2\alpha+2} \sum_{k=1}^n\tau_k e^{(2\alpha+3) t_{k-1}}.
\end{align*}
This is nice, insofar as the right hand side is a lower sum of the integral $(1+\tau)^{2\alpha+2}\int_0^{t_n}e^{(2\alpha+3) s}\dd s$, hence
\begin{equation}\begin{split}\label{eq:sntn}
	&\sh_n \leq (1+\tau)^{2\alpha+2}(2\alpha+3)^{-1}\big[e^{(2\alpha+3)t_n}-1\big] \\
	\Longrightarrow\;&
	e^{-t_n} \leq \big(1 + a_\tau\sh_n(2\alpha+3)\big)^{-1/(2\alpha+3)},
\end{split}\end{equation}
with $a_\tau = (1+\tau)^{-(2\alpha+2)}$ converging to $1$ as $\tau\to0$.
For a given solution $\xvec_\Delta$ of \eqref{eq:dmm3} with $\lambda=1$ and fixed discretization $\Delta=(\thet;\theh)$,
it is a trivial task to check, that the recursively defined sequence of vectors $\ds_{\St^n}\xvec_\Delta^n$
is a solution to \eqref{eq:dmmds} for $S=\St^{n-1}$, $R=\St^n$, $\lt=0$ and $\taut=\taut_n$ defined in \eqref{eq:partitiontaut}.
Henceforth, we write $\xvec_\Deltat^n = \ds_{\St^n}\xvec_\Delta^n$ with the discretization $\Deltat=(\thett;\theh)$
We can hence use the discrete scaling property of $\Halz$ and invoke \eqref{eq:expHn} of Lemma \ref{lem:exp}, then
\begin{equation}\begin{split}\label{eq:better}
	&(1+2\tau_n)(\St^n)^{\frac{2\alpha-1}{2}}\big(\Haiz(\xvec_\Deltat^n) - \Haiz(\bvec_{\Delta,\alpha,0}^n)\big) 
	\leq (\St^{n-1})^{\frac{2\alpha-1}{2}}\big(\Haiz(\xvec_\Deltat^{n-1}) - \Haiz(\bvec_{\Delta,\alpha,0}^{n-1})\big) \\
	\Longrightarrow\;&
	(1+2\tau_n)(1+\tau_n)^{\frac{2\alpha-1}{2}}\big(\Haiz(\xvec_\Deltat^n) - \Haiz(\bvec_{\Delta,\alpha,0}^n)\big) 
	\leq \Haiz(\xvec_\Deltat^{n-1}) - \Haiz(\bvec_{\Delta,\alpha,0}^{n-1}) \\
	\Longrightarrow\;&
	(1+2\tau_n)\big(\Haiz(\xvec_\Deltat^n) - \Haiz(\bvec_{\Delta,\alpha,0}^n)\big) 
	\leq \Haiz(\xvec_\Deltat^{n-1}) - \Haiz(\bvec_{\Delta,\alpha,0}^{n-1}),
\end{split}\end{equation}
where we used in the last step $(1+\tau_n)>1$.
As before in the proof of \eqref{eq:expH} of Theorem \ref{thm:main2}, this yields for any $n=0,1,\ldots$,
due to \eqref{eq:sntn}
\begin{align*}
	\Haiz(\xvec_\Deltat^n) - \Haiz(\bvec_{\Delta,\alpha,0}^n) 
	&\leq \big(\Haiz(\xvec_\Deltat^0) - \Haiz(\xvecm)\big) e^{-\frac{2t_n}{1+2\tau}} \\
	&\leq \big(\Haiz(\xvec_\Deltat^0) - \Haiz(\xvecm)\big) \big(1 + a_\tau\sh_n(2\alpha+3)\big)^{-\frac{2}{b_\tau(2\alpha+3)}},
\end{align*}
with $b_\tau = 1+2\tau$.
Theorem \ref{thm:main3} follows, using $\bantn=\cf_\theh[\bvec_{\Delta,\alpha,0}]$ and a Csiszar-Kullback inequality, see \cite[Theorem 30]{CaJuMa}.
\end{proof2}

%
%
\section{Numerical experiments}\label{sec:num}
%
%
\subsection{Non-uniform meshes}
An equidistant mass grid --- as used in the analysis above --- leads to a good spatial resolution of regions where the value of $u^0$ is large,
but provides a very poor resolution in regions where $u^0$ is small.
Since we are interested in regions of low density, and especially in the evolution of supports,  
it is natural to use a \emph{non-equidistant} mass grid with an adapted spatial resolution, like the one defined as follows:
The mass discretization of $[0,M]$ is determined by a vector $\vech=(\xi_0,\xi_1,\xi_2,\ldots,\xi_{K-1},\xi_K)$,
with $0=\xi_0 < \xi_1 < \cdots < \xi_{K-1} < \xi_K = M$
and we introduce accordingly the distances (note the convention $\xi_{-1} = \xi_{K+1} = 0$)
\begin{align*}
  \delta_\kappa = \xi_\kappp-\xi_\kappm, 
  \quad\text{and}\quad 
  \delta_k = \frac12(\delta_\kph+\delta_\kmh) = \frac12(\xi_{k+1}-\xi_{k-1})
\end{align*}
for $\kappa\in\hval$ and $k\in\ival$, respectively.
The piecewise constant density function $u\in\densNN$ corresponding to a vector $\xvec\in\setR^{K-1}$
is now given by
\begin{align*}
  u(x) = z_\kappa \quad\text{for $x_\kappm<x<x_\kappp$}, \quad
  \text{with}\quad z_\kappa = \frac{\delta_\kappa}{x_\kappp-x_\kappm}.
\end{align*}
The Wasserstein-like metric (and its corresponding norm) needs to be adapted as well:
the scalar product $\spr{\cdot}{\cdot}$ is replaced by
\begin{align*}
  \langle\vvec,\wvec\rangle_\vech = \sum_{k\in\ival} \delta_kv_kw_k
	\quad\textnormal{and}\quad \nrm{\vvec} = \langle\vvec,\vvec\rangle_\vech.
\end{align*}
Hence the metric gradient $\nabla_\vech f(\xvec)\in\setR^{K+1}$ of a function $f:\xseqNN\to\setR$ at $\xvec\in\xseqNN$
is given by
\begin{align*}
  \big[\nabla_\vech f(\xvec)\big]_k = \frac1{\delta_k}\partial_{x_k}f(\xvec).
\end{align*}
Otherwise, we proceed as before:
the entropy is discretized by restriction, and the discretized information functional is the self-dissipation of the discretized entropy.
Explicitly, the resulting fully discrete gradient flow equation
\begin{align*}
  \frac{\xvec_\Delta^n-\xvec_\Delta^{n-1}}\tau = - \nabla_\vech\Falz(\xvec_\Delta^n)
\end{align*}
attains the form
\begin{equation}\begin{split}\label{eq:nonuniform}
  \frac{x^n_k-x^{n-1}_k}{\tau_n} 
  = \frac{2\alpha}{(2\alpha+1)^2\delta_k}\left[
    (z^n_\kph)^{\alpha+\frac{3}{2}}[\operatorname{D}_\vech^2(\zvec^n)^{\alpha+\frac{1}{2}}]_\kph
    -(z^n_\kmh)^{\alpha+\frac{3}{2}}[\operatorname{D}_\vech^2(\zvec^n)^{\alpha+\frac{1}{2}}]_\kmh
  \right]
	+ \lambda x_k^n,
\end{split}\end{equation}
with $[\operatorname{D}_\vech^2\zvec^{\alpha+\frac{1}{2}}]_\kmh:=(z_{\kph}^{\alpha+\frac{1}{2}}-2z_\kmh^{\alpha+\frac{1}{2}}+z_\kmd^{\alpha+\frac{1}{2}})/\delta_k^2$

%
\subsection{Implementation}
To guarantuees the existence of an initial vector $\xvec_\Delta^0\in\xseqN$, 
which "reaches" any mass point of $u^0$, i.e.  
$[x_0^0,x_K^0]\subseteq\operatorname{supp}(u^0)$,
one has to consider initial density functions $u^0$ with compact support.

Starting from the initial condition $\xvec_\Delta^0$, the fully discrete solution is calculated inductively 
by solving the implicit Euler scheme \eqref{eq:nonuniform} for $\xvec_\Delta^n$, given $\xvec_\Delta^{n-1}$.
In each time step, a damped Newton iteration is performed, with the solution from the previous time step as initial guess.

%
\subsection{Experiment I -- Exponential decay rates}
\begin{figure}[t]
  \centering
  \subfigure{\includegraphics[width=0.45\textwidth]{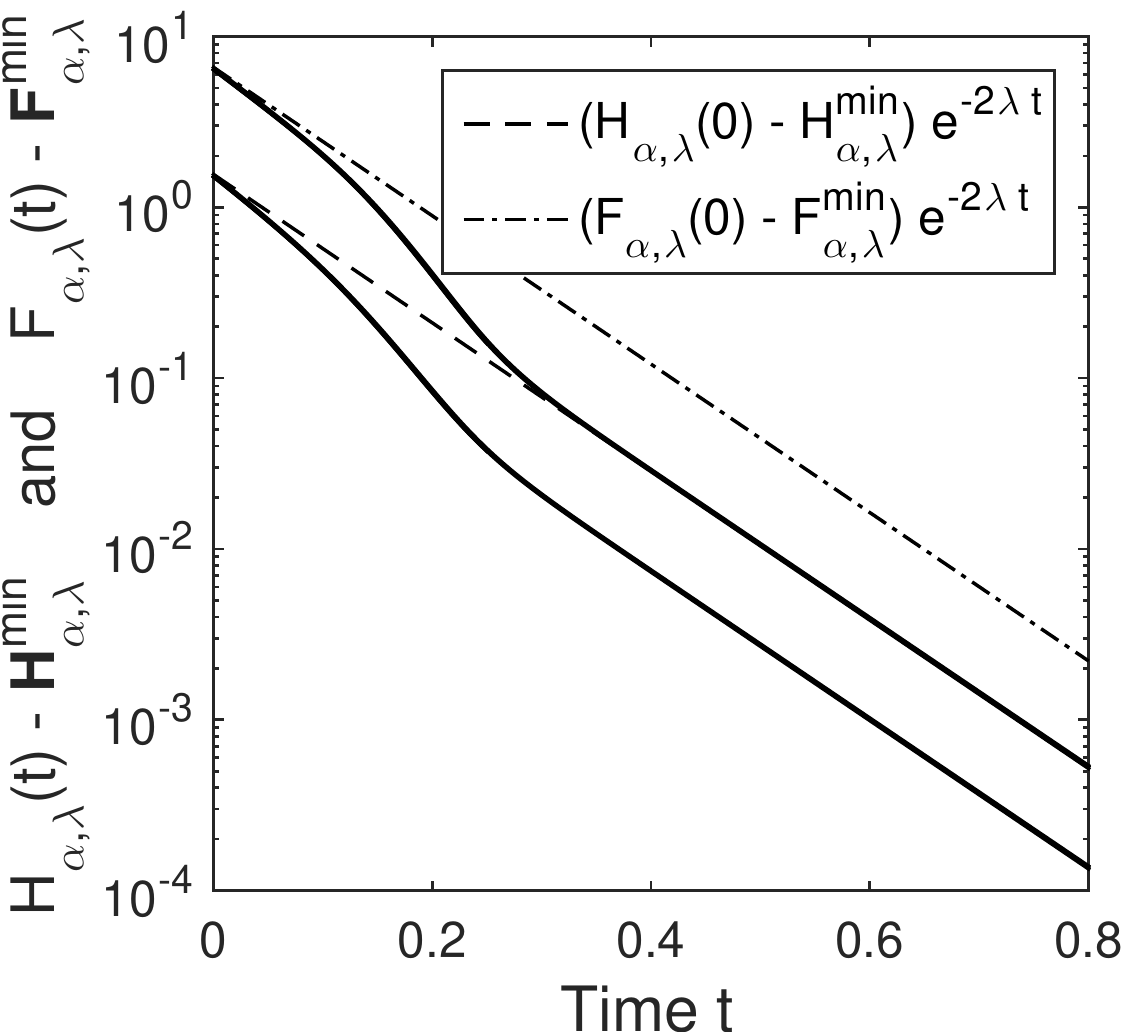}}
  \subfigure{\includegraphics[width=0.42\textwidth]{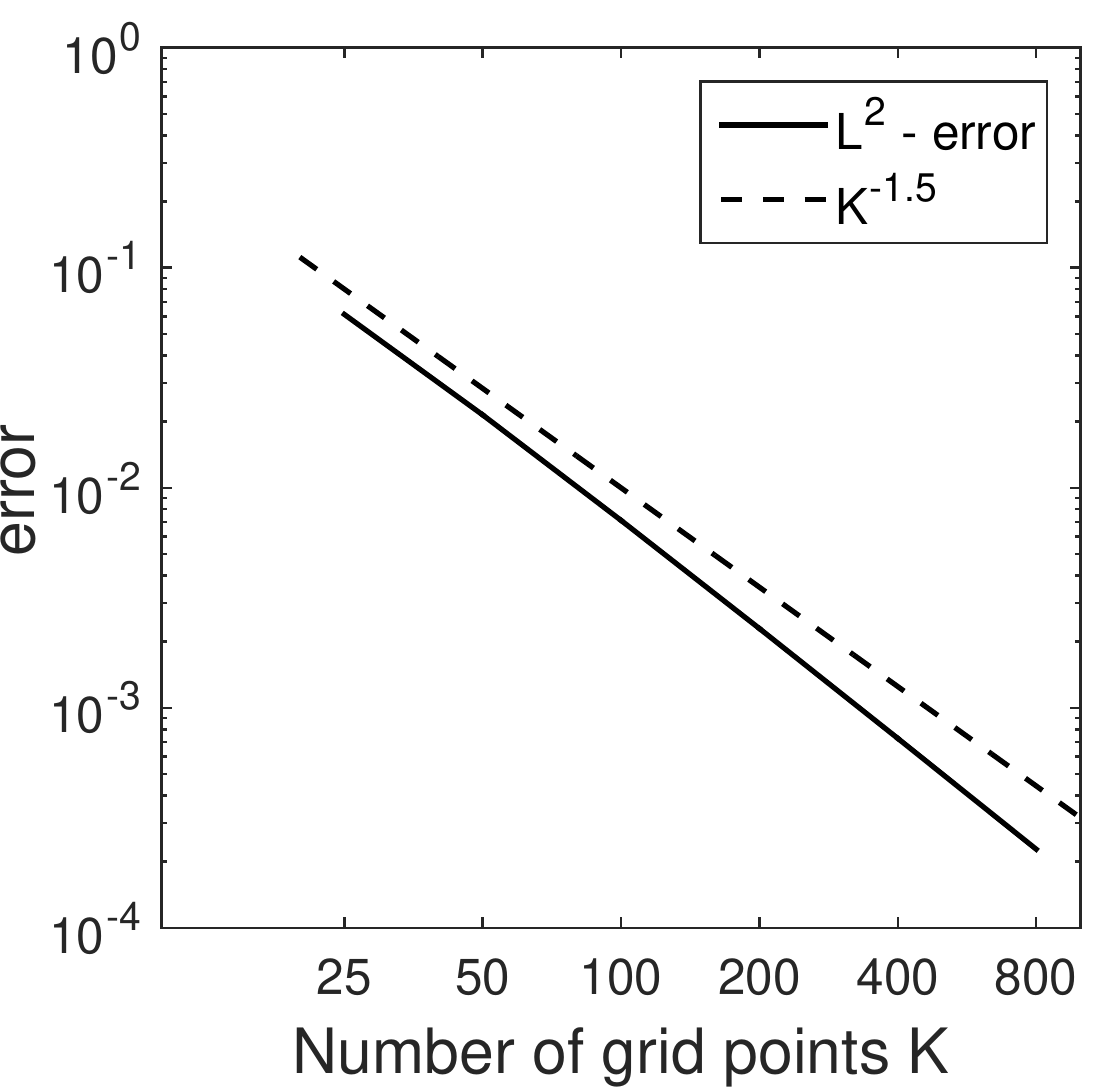}}
  \caption{\emph{Left}: Numerically observed decay of $H_{\alpha,\lambda}(t)-\Halzm$ and $F_{\alpha,\lambda}(t)-\Falzm$ along a time period of $t\in[0,0.8]$,
						using $K=25,50,100,200$, in comparison to the upper bounds 
						$(\Hal(u^0)-\Hal(\bal))\exp(-2\lambda t)$ and $(\Fal(u^0)-\Fal(\bal))\exp(-2\lambda t)$, respectively.
					\emph{Right}: Convergence of discrete minimizers $\um$ with a rate of $K^{-1.5}$.}
  \label{fig:fig2}
\end{figure}

As a first numerical experiment, we want to analyse the rate of decay in case of positive confinement $\lambda=5$, using $\alpha=1$.
For that purpose, consider the initial density function
\begin{align}\label{eq:u0exp}
	u^0 = \begin{cases}  0.25|\sin(x)|\cdot(0.5+\indy_{x>0}(x)) ,& x\in[-\pi,\pi],\\ 0 ,&\textnormal{esle}\end{cases},
\end{align}
Figure \ref{fig:fig1} shows the evolution of the discrete density $u_\Delta$ at times $t= 0.05,0.1,0.15,0.175,0.225$, using $K=200$. 
The two initially seperated clusters quickly merge, and finally changes the shape towards a Barenblatt-profile (doted line).

The exponential decay of the entropies $\Halz$ and $\Falz$ along the solution can be seen in figure \ref{fig:fig2}/left
for $K=25,50,100,200$, where we observed the evolution for $t\in[0,0.8]$. 
Note that we write $H_{\alpha,\lambda}(t)=\Halz(\xvec_\Delta^n)$ and $F_{\alpha,\lambda}(t)=\Falz(\xvec_\Delta^n)$ for $t\in(t_{n-1},t_n]$,
and set $H_{\alpha,\lambda}(0)=\Halz(\xvec_\Delta^0)$ and $F_{\alpha,\lambda}(0)=\Falz(\xvec_\Delta^0)$.
As the picture shows, the rate of decay does not really depent on the choice of $K$,
in fact the curves lie de facto on the top of each other.
Furthermore, the curves are bounded from above by $(\Hal(u^0)-\Hal(\bal))\exp(-2\lambda t)$ 
and $(\Fal(u^0)-\Fal(\bal))\exp(-2\lambda t)$ at any time, respectively, as \eqref{eq:expH} \& \eqref{eq:expF} from 
Theorem \ref{thm:main2} postulate. 
One can even recognize, that the decay rates are even bigger at the beginning, until the moment $t=...$ 
when $u_\Delta$ finishes its "fusion" to one single Barenblatt-like curve. 
After that, the solution's evolution mainly consists of a transveral shift towards the stationary solution $\bal$, 
which is reflected by a henceforth constant rate of approximately $-2\lambda$.

Moreover, figure \ref{fig:fig3}/right pictures the converence of $\um$ towards $\bal$. We used several values for 
the spatial discretization, $K=25,50,100,200,400,800$,
and plotted the $L^2$-error. The observed rate of convergence is $K^{-1.5}$.

%
\subsection{Experiment II -- Self-similar solutions}
A very interesting consequence of section \ref{sec:confn} is,
that the existence of self-similiar solutions bequeath from the continuous to the discrete case. 
In more detail, this means the following: 
Set $\lambda=0$ and define for $t\in[0,+\infty)$
\begin{align}\label{eq:selfsim}
	\bant(t,\cdot) := \ds_{R(t)} \bai,\quad\textnormal{with}\quad R(t) := \big(1+(2\alpha+3)t\big)^{1/(2\alpha+3)},
\end{align}
then $\bant$ is a solution of the continuous problem \eqref{eq:fofo0} with $u^0=\bant(0,\cdot)$.
In the discrete setting, solutions to \eqref{eq:nonuniform} with $\lambda=0$
are inductively given by 
an initial vector $\bvec_{\Delta,\alpha,0}^0$ with corresponding density $u_\Delta^0=\cf_\vech[\bvec_{\Delta,\alpha,0}^0]$
that approaches $\bant(0,\cdot)$,
and $\bvec_{\Delta,\alpha,0}^n=\ds_{\St^n}\bvec_{\Delta,\alpha,0}^0$ with $\St^n$ defined as in \eqref{eq:Sn},
for further $n=1,2\ldots,$.

As figure \ref{fig:fig3} shows, the resulting sequence of densities $u_\Delta$ (black lines) approaches the continuous solution $\bant$ of 
\eqref{eq:selfsim} (red lines)
astonishingly well, even if the discretization parameters are choosen quite rough. 
In this specific case we used $K = 50$ and $\tau = 10^{-3}$. 
The discrete and continuous solutions are evaluated at times $t=0,0.1,1,10,100$

\begin{figure}
  \centering
	\includegraphics[scale = 0.65]{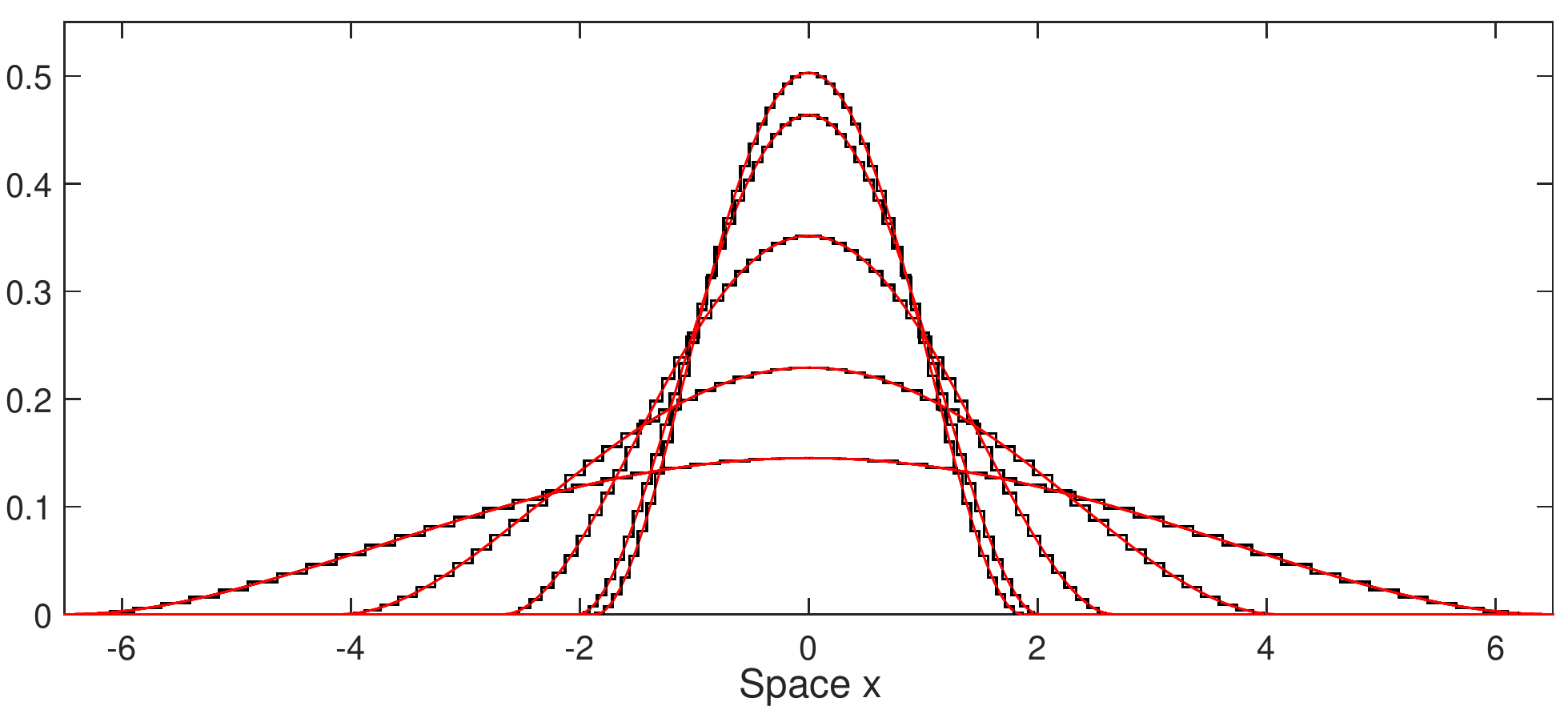}
  \caption{Snapshots of the densities $\bant(t,\cdot)$ (red lines) and $u_\Delta$ (black lines) for the initial condition $\bant(0,\cdot)$
	at times $t=0$ and $t= 0.1\cdot 10^{i}$, $i=0,\ldots,3$, 
	using $K=50$ grid points and the time step size $\tau=10^{-3}$.}
  \label{fig:fig3}
\end{figure}

\vspace{0.3cm}
\begingroup 
\begin{center}
	\begin{minipage}[t]{0.33\textwidth}
		\begin{center}
		\includegraphics[scale=0.4]{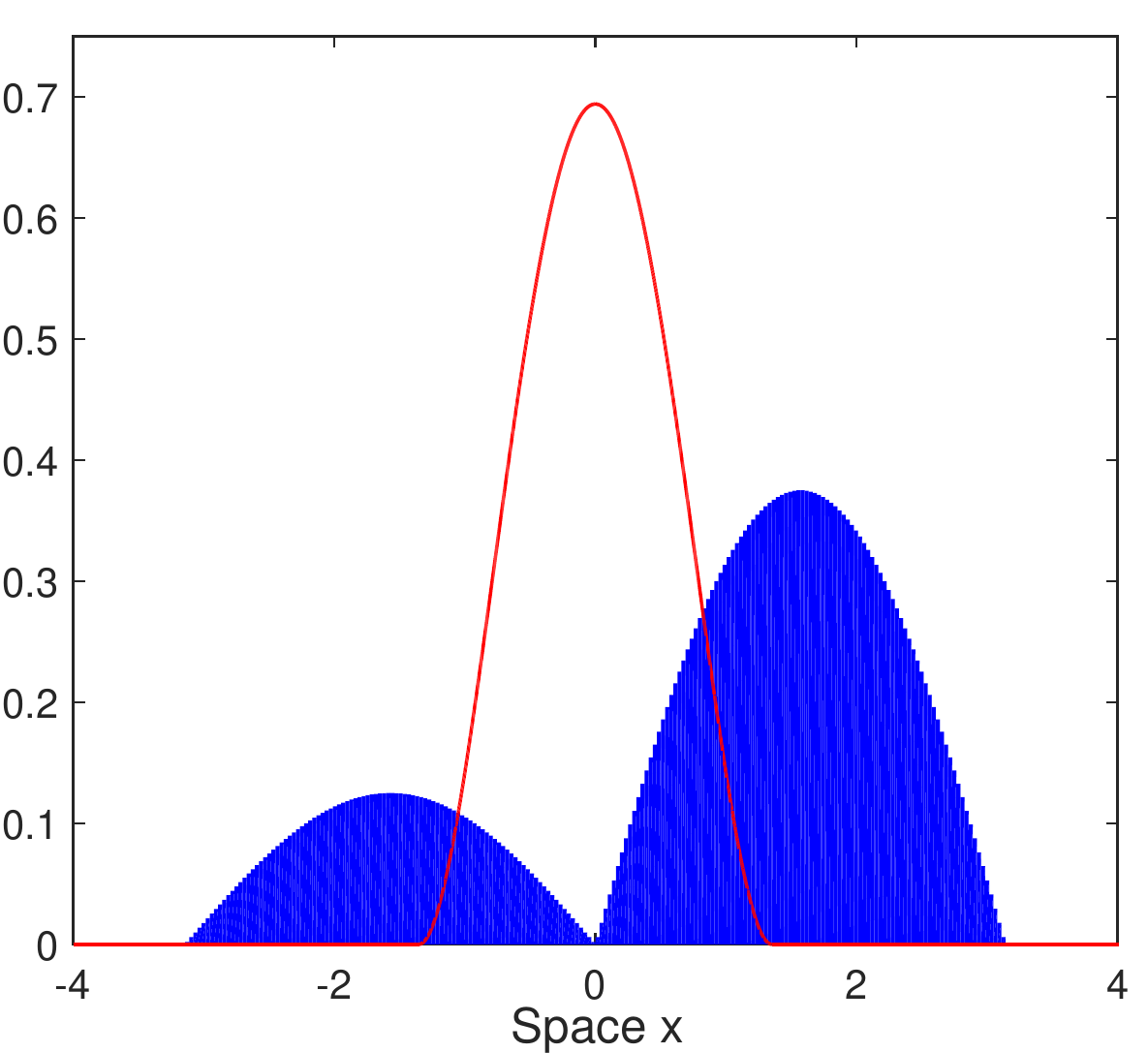}
		\end{center}
	\end{minipage}%
	\begin{minipage}[t]{0.33\textwidth}
		\begin{center}
		\includegraphics[scale=0.4]{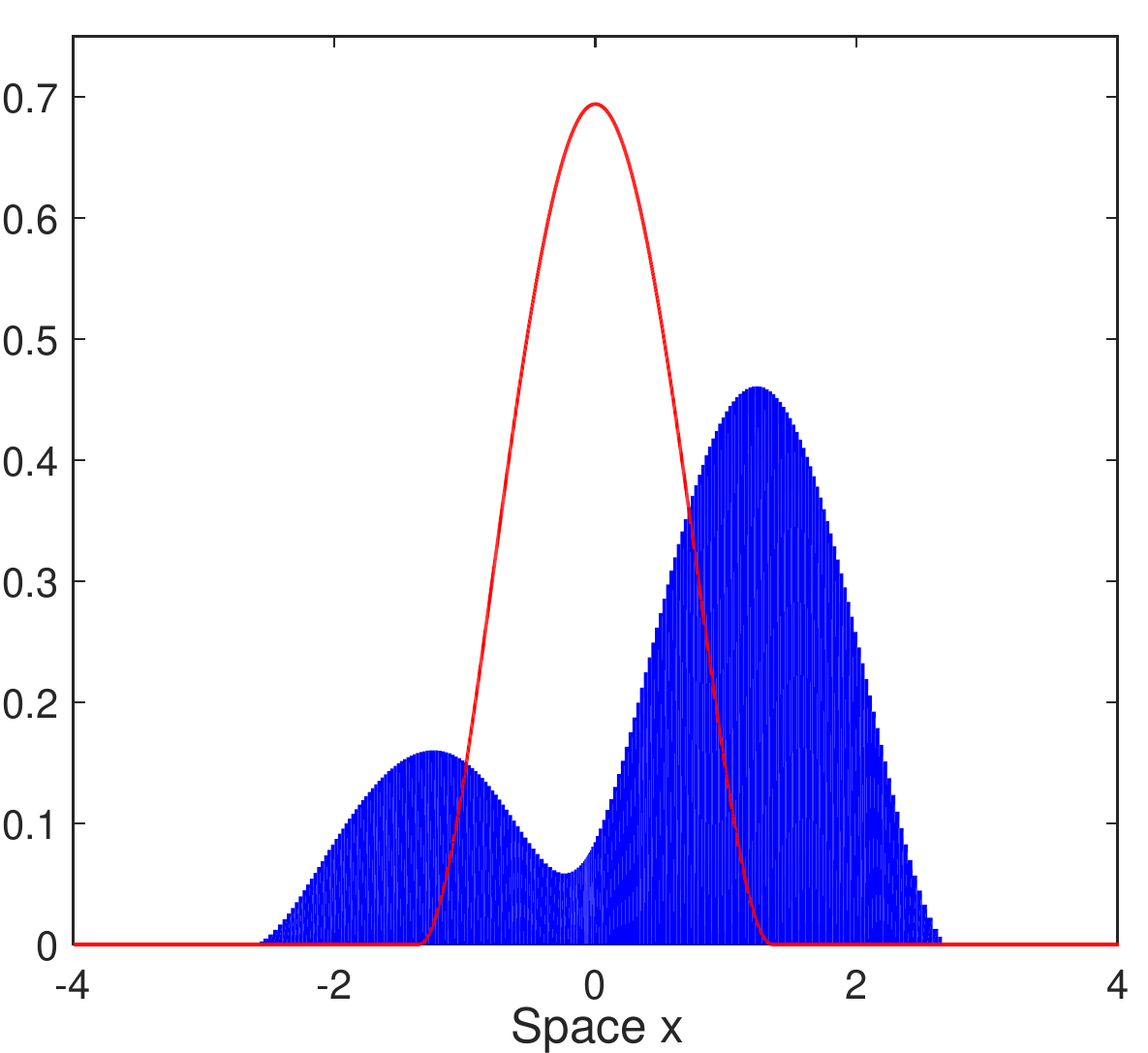}
		\end{center}
	\end{minipage}%
	\begin{minipage}[t]{0.33\textwidth}
		\begin{center}
		\includegraphics[scale=0.4]{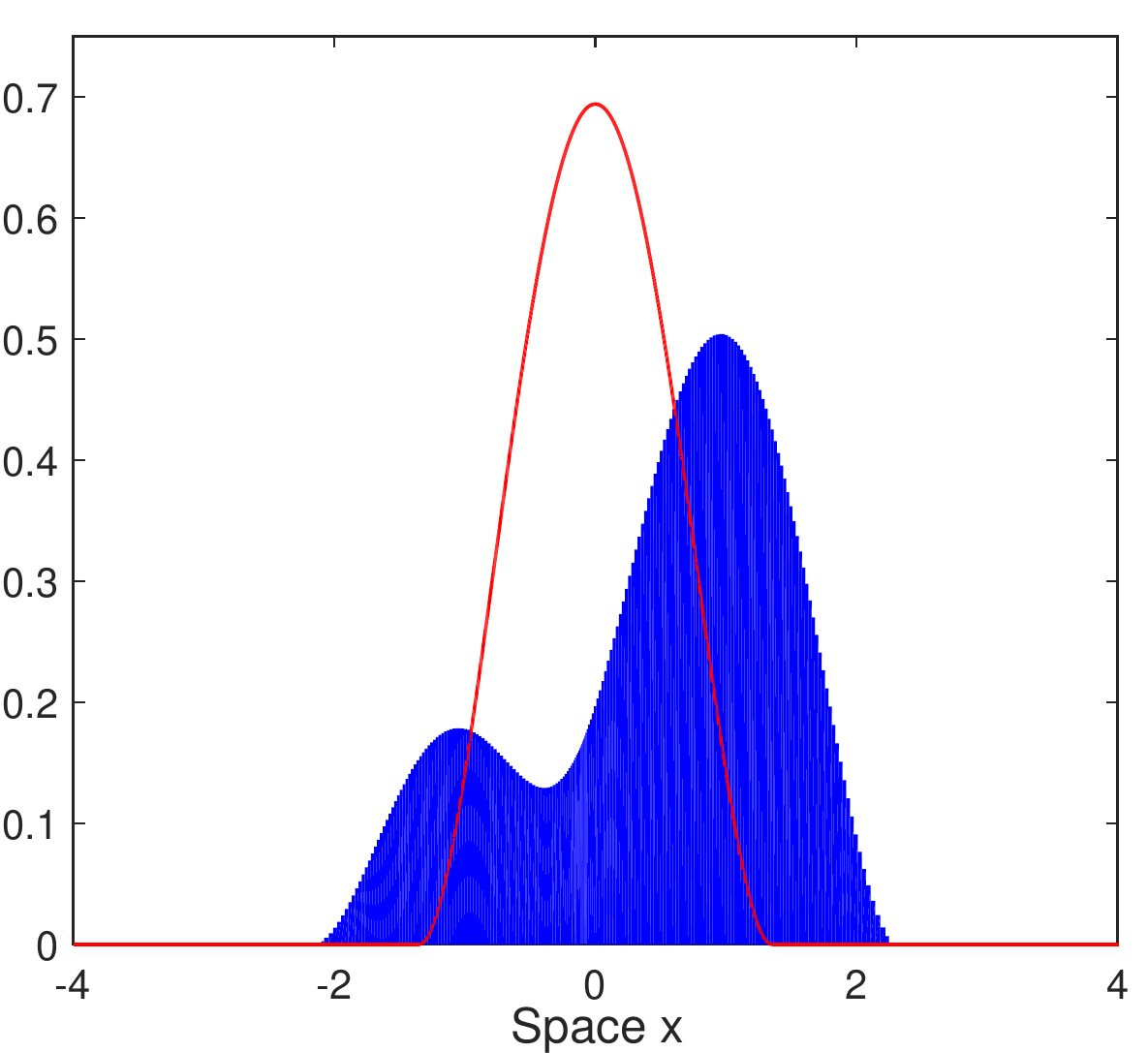}
		\end{center}
	\end{minipage}
	\begin{minipage}[t]{0.33\textwidth}
		\begin{center}
		\includegraphics[scale=0.4]{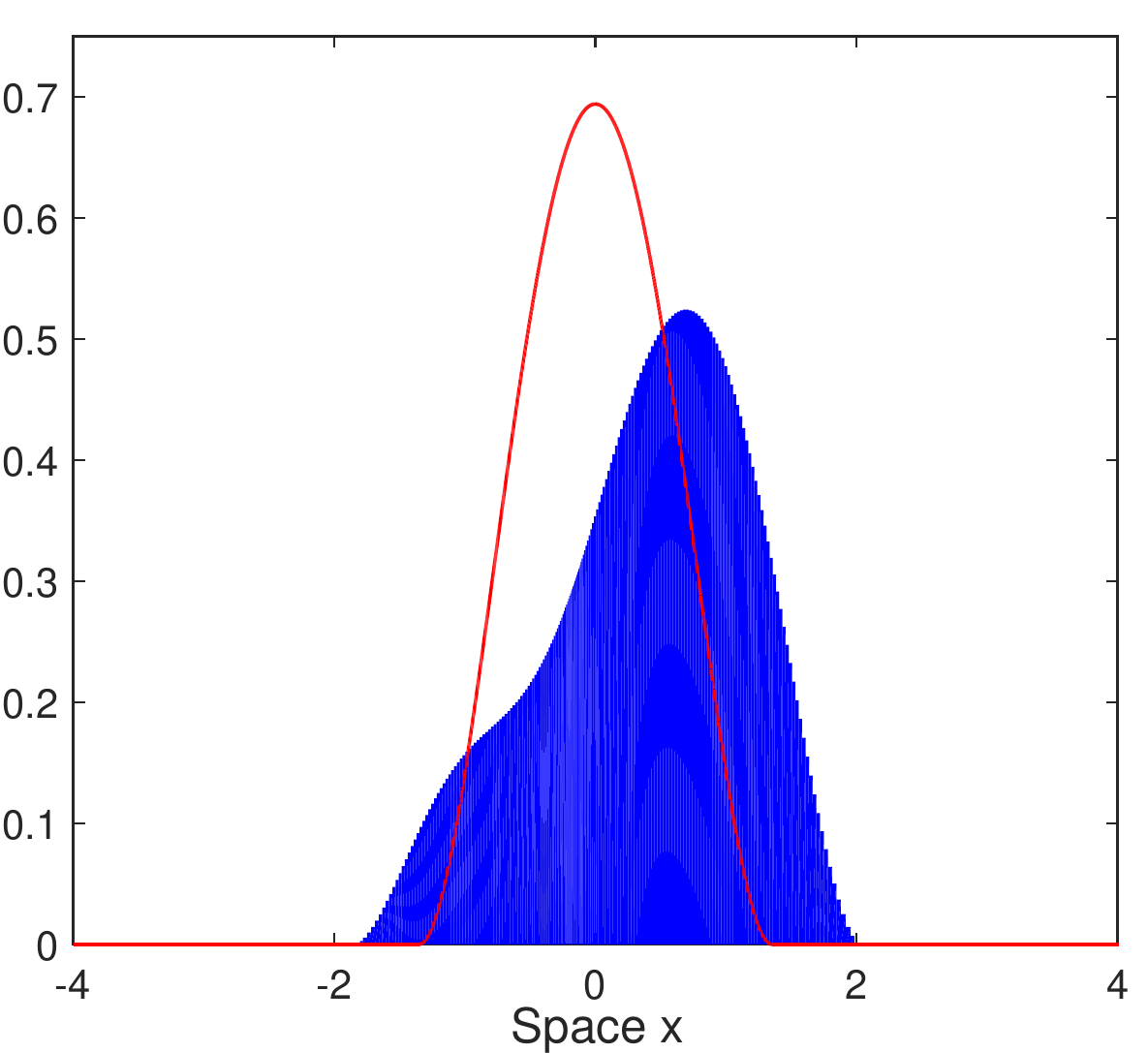}
		\end{center}
	\end{minipage}%
		\begin{minipage}[t]{0.33\textwidth}
		\begin{center}
		\includegraphics[scale=0.4]{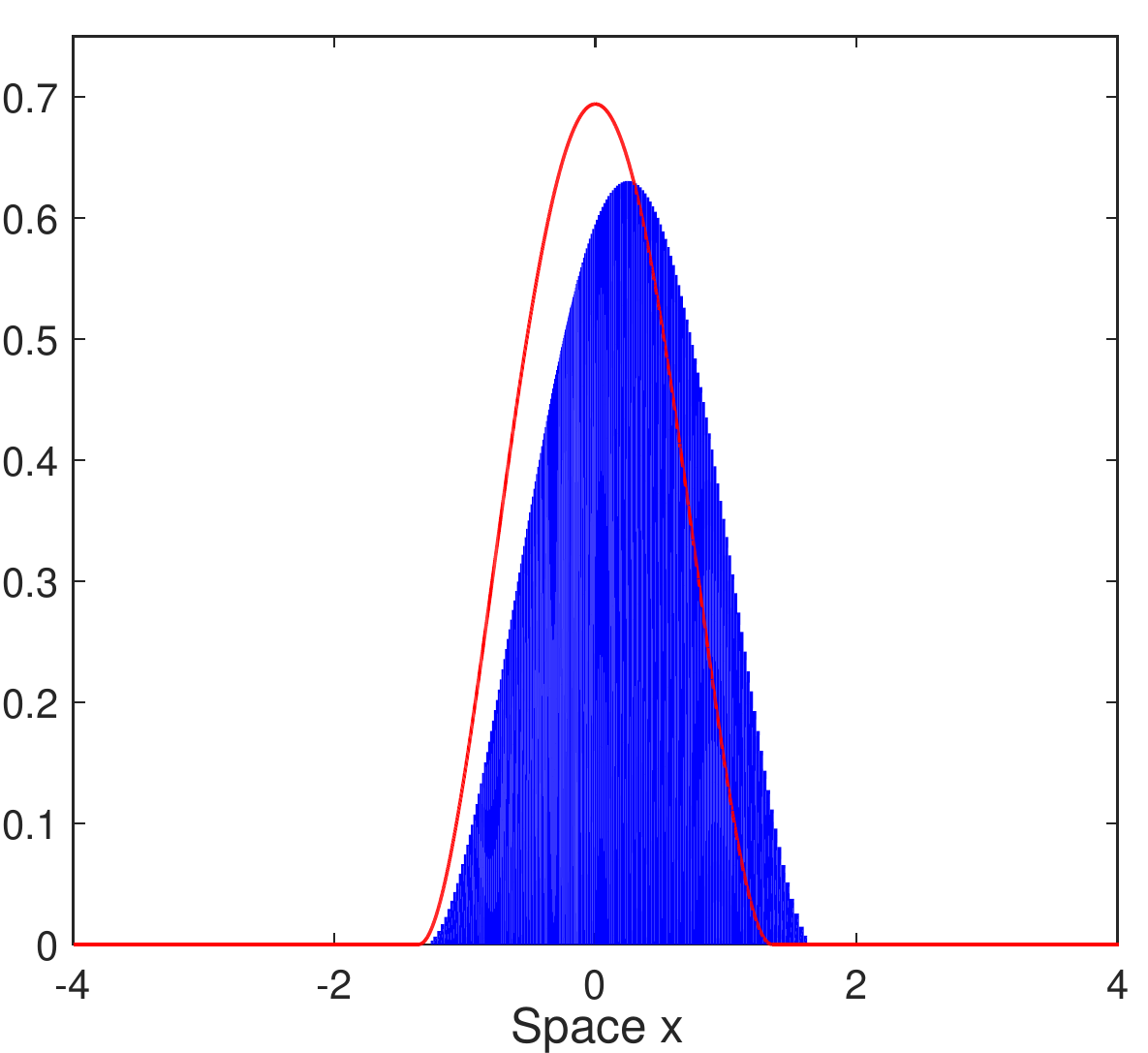}
		\end{center}
	\end{minipage}%
		\begin{minipage}[t]{0.33\textwidth}
		\begin{center}
		\includegraphics[scale=0.4]{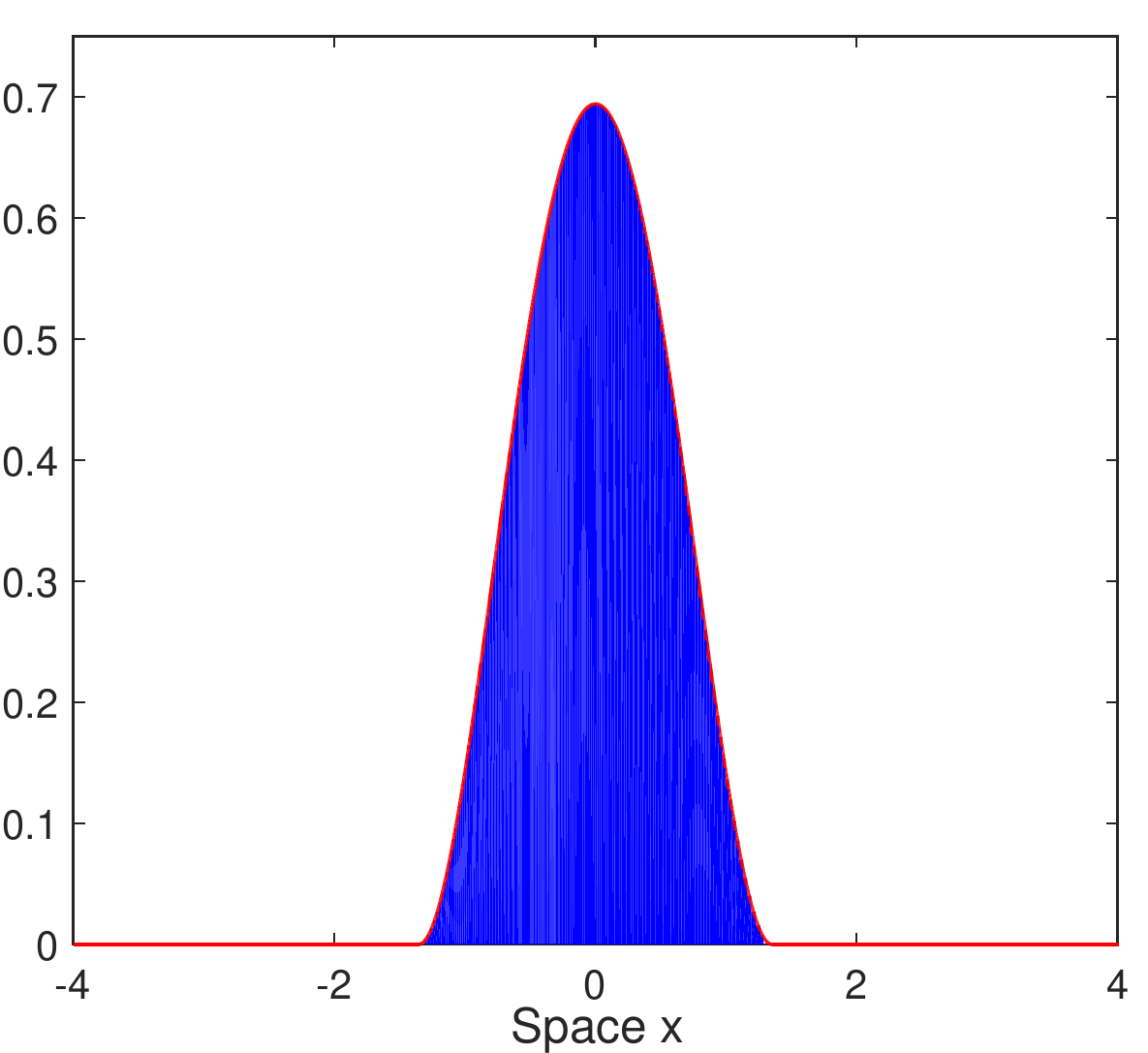}
		\end{center}
	\end{minipage}
\end{center}
\captionof{figure}{Evolution of a discrete solution $u_\Delta$, evaluated at different times $t = 0,0.05,0.1,0.15,0.175,0.25$ (from top left to bottom right).
									The red line is the corresponding Barenblatt-profile $\bal$.}
\label{fig:fig1}
\endgroup
\vspace{0.3cm}

\bibliography{Horst}
\bibliographystyle{siam}

\end{document}